\documentclass[a4paper,UKenglish,cleveref,autoref,thm-restate]{lipics-v2021}
\usepackage[disable]{todonotes} %

\usepackage{tikz}
\usepackage{tikzconfig}
\usepackage{tikzmacros}

\usepackage{algorithm}
\usepackage[noend]{algpseudocode}
\usepackage{subcaption}
\usepackage[quotation,notion,electronic]{knowledge}
 \definecolor{Blue Sapphire}{HTML}{003050} 
\definecolor{VeryDarkBlue}{HTML}{001050} 
\definecolor{Gamboge}{HTML}{ee9b00}
\definecolor{Ruby Red}{HTML}{ab2226}

\IfKnowledgePaperModeTF{
}{
	\knowledgestyle{intro notion}{color={Ruby Red}, emphasize}
	\knowledgestyle{notion}{color={VeryDarkBlue}}
	\hypersetup{
		colorlinks=true,
		breaklinks=true,
		linkcolor={Blue Sapphire}, %
		citecolor={Blue Sapphire}, %
		filecolor={Blue Sapphire}, %
		urlcolor={Blue Sapphire},
	}
}
\IfKnowledgeCompositionModeTF{
	\knowledgeconfigure{anchor point color={Ruby Red}, anchor point shape=corner}
	\knowledgestyle{intro unknown}{color={Gamboge}, emphasize}
	\knowledgestyle{intro unknown cont}{color={Gamboge}, emphasize}
	\knowledgestyle{kl unknown}{color={Gamboge}}
	\knowledgestyle{kl unknown cont}{color={Gamboge}}
}{
}

 \input{popmdps.kl}

\nolinenumbers
\usepackage{xspace}
\usepackage{macros}
\usepackage{subcaption}

\title{Optimal Sequential Flows}%

\author{Hugo Gimbert}{CNRS, LaBRI, Université de Bordeaux, France}{hugo.gimbert@labri.fr}{}{}{}

\author{Corto Mascle}{MPI-SWS, Kaiserslautern, Germany}{cmascle@mpi-sws.org}{}{}{}

\author{Patrick Totzke}{University of Liverpool, UK}{totzke@liverpool.ac.uk}{}{}{}

\authorrunning{H.~Gimbert, C.~Mascle, P.~Totzke} %

\Copyright{ CC-BY;  http://creativecommons.org/licenses/by/3.0/} %



\begin{document}
\maketitle
\begin{abstract}
	We provide a new algebraic technique to solve the sequential flow problem in polynomial space. The task is to maximise the flow through a graph where edge capacities can be changed over time
by choosing a sequence of capacity labelings from a given finite set.
Our method is based on a novel factorization theorem for finite semigroups that, applied to a suitable flow semigroup, allows to derive small witnesses.
This generalises to multiple in/output vertices, as well as regular constraints.

\end{abstract}

\newpage

%\newpage
%\setcounter{page}{1}
\section{Introduction}

Determining the maximal flow through a network under capacity constraints is a classical optimisation problem \cite{ford1956maximal}.
The \sfp\ is a dynamic variant in which the labelling of edges by capacities
is not static but can change over (discrete) time.
We are given a finite set $A\subseteq (\+N\cup\{\omega\})^{V\x V}$ of which each element $a\in A$ prescribes capacities for every edge in the 
directed graph\footnote{A capacity / flow value $\omega$ means unbounded, i.e., finite but arbitrarily large.}.
Every length-$\ell$ capacity word $a_1a_2\cdots a_\ell \in A^*$ uniquely determines a \emph{pipeline}, a graph of size $(\ell+1)\cdot \card{V}$ together with edge capacities where at time $1\le i\le \ell$, the edge $v\to v'$ has capacity $a_i(v,v')$.
The \sfp\ asks to determine the supremum of flow values through any such pipeline.

Notice that there is no limit on the length $\ell$ of the capacity words and therefore, the optimal sequential flow can be unbounded even if there is no finite word witnessing this.

\begin{example}\label{ex:intro}
	Consider the graph with vertices $V=\{v_1,v_2,v_3,v_4\}$,
	with source $v_s=v_1$ and target $v_t=v_4$,
	and capacity constraints $A=\{a,b\}$
	as depicted on the left in Figure~\ref{fig:ex1}.
In both capacities $a$ and $b$, there is no path from the source $v_1$ to the target $v_4$. It is therefore necessary to combine them sequentially in order to enable positive flow from $v_1$ to $v_4$. This can be achieved using the capacity word $abba$, which has a maximal flow value $2$.

\begin{figure}[t]
  \centering
  \newlength{\subfigheight}
  \setlength{\subfigheight}{2.5cm} %
  \begin{subfigure}[b][\subfigheight][b]{0.16\textwidth}
    \centering
    \vspace*{\fill}
    \begin{tikzpicture}[node distance=1.2cm]

\node[state] (s) at (0,0) {$v_1$};
\node[state,right of=s] (u) {$v_2$};
\node[state,below of=s] (v) {$v_3$};
\node[state,right of=v] (t) {$v_4$};

\path[use as bounding box] (s) rectangle (t);
\draw[extracolour1] (s) edge node[above] {$\omega$} (u);
\draw[extracolour1] (v) edge node {$\omega$} (u);
\draw[extracolour1,->] (v) edge node[above] {$\omega$} (t);
\end{tikzpicture}
    \vspace*{\fill}
    \caption*{\centering capacity $a$}
  \end{subfigure}
  \hfill
  \begin{subfigure}[b][\subfigheight][b]{0.18\textwidth}
    \centering
    \vspace*{\fill}
    \begin{tikzpicture}[node distance= 1.2cm]

\node[state] (s) at (0,0) {$v_1$};
\node[state,right of=s] (u) {$v_2$};
\node[state,below of=s] (v) {$v_3$};
\node[state,right of=v] (t) {$v_4$};

\path[use as bounding box] (s) rectangle (t);
\draw[extracolour2,->] (v) edge[loop left] node[extracolour2,left] {$\omega$} (v);
\draw[extracolour2,->] (u) edge[loop right] node[extracolour2,right] {$\omega$} (u);
\draw[extracolour2,->] (u) edge[] node[extracolour2,above] {$1$} (v);
\end{tikzpicture}
    \vspace*{\fill}
    \caption*{\centering capacity $b$}
  \end{subfigure}
  \hfill
  \begin{subfigure}[b][\subfigheight][b]{0.27\textwidth}
    \centering
    \vspace*{\fill}
    {
      \begin{tikzpicture}[scale=0.47]
	\flowlabel[xshift=-0.2cm]{4}{1/$v_1$,2/$v_2$,3/$v_3$,4/$v_4$}
        \flowwithletter[extracolour1]{$a$}{4}{1-2/$\om$,3-2/$\om$,3-4/$\om$}
        \flowwithletter[extracolour2,xshift=1.5cm]{$b$}{4}{2-2/$\om$,2-3/$1$,3-3/$\om$}
        \flowwithletter[extracolour2,xshift=3cm]{$b$}{4}{2-2/$\om$,2-3/$1$,3-3/$\om$}
        \flowwithletter[extracolour1,xshift=4.5cm]{$a$}{4}{1-2/$\om$,3-2/$\om$,3-4/$\om$}
      \end{tikzpicture}
    }
    \vspace*{\fill}
    \caption*{\centering pipeline $abba$}
  \end{subfigure}
  \hfill
  \begin{subfigure}[b][\subfigheight][b]{0.27\textwidth}
    \centering
    \vspace*{\fill}
    {
      \begin{tikzpicture}[scale=0.47]
	\flowlabel[xshift=-0.2cm]{4}{1/$v_1$,2/$v_2$,3/$v_3$,4/$v_4$}
        \flowwithletter[extracolour1]{$a$}{4}{1-2/$2$}
        \flowwithletter[extracolour2,xshift=1.5cm]{$b$}{4}{2-2/$1$,2-3/$1$}
        \flowwithletter[extracolour2,xshift=3cm]{$b$}{4}{2-3/$1$,3-3/$1$}
        \flowwithletter[extracolour1,xshift=4.5cm]{$a$}{4}{3-4/$2$}
      \end{tikzpicture}
    }
    \vspace*{\fill}
    \caption*{\centering max flow for $abba$}
  \end{subfigure}
  \caption{Two capacity constraints, %
  a pipeline, and an optimal flow through it with value $2$.}
  \label{fig:ex1}
\end{figure}

Even more flow can be transported from $v_1$ to $v_4$ through longer pipelines.
	For every $n>0$ the pipeline for capacity word $ab^na$ has a flow of (maximal) value $n$, as depicted in Figure~\ref{fig:abnavaluen}.
	The \emph{sequential} flow for $A$ is therefore unbounded. %

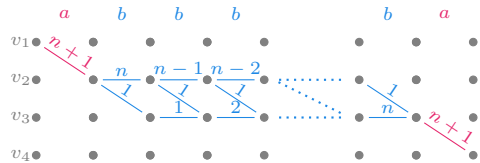
\begin{figure}[t]
\begin{center}
	\begin{tikzpicture}[scale=0.5]

	  \flowlabel[xshift=-0.2cm]{4}{1/$v_1$,2/$v_2$,3/$v_3$,4/$v_4$}
	  \flowwithletter[extracolour1]{$a$}{4}{1-2/$n+1$}
	  \flowwithletter[extracolour2,xshift=1.5cm]{$b$}{4}{2-2/$n$,2-3/$1$}
	  \flowwithletter[extracolour2,xshift=3cm]{$b$}{4}{2-2/$n-1$,2-3/$1$,3-3/$1$}
	  \flowwithletter[extracolour2,xshift=4.5cm]{$b$}{4}{2-2/$n-2$,2-3/$1$,3-3/$2$}

	  \path[extracolour2] (6.25,-1) edge[dotted,thick,-] (8.25,-1);
	  \path[extracolour2] (6.25,-1) edge[dotted,thick,-] (8.25,-2);
	  \path[extracolour2] (6.25,-2) edge[dotted,thick,-] (8.25,-2);

	  \flowwithletter[extracolour2,xshift=8.5cm]{$b$}{4}{2-3/$1$,3-3/{$n$}}

	  \flowwithletter[extracolour1,xshift=10cm]{$a$}{4}{3-4/$n+1$}
	\end{tikzpicture}
	\end{center}
	\caption{A flow of value $n+1$ through the pipeline $ab^{n+1}a$.}
\label{fig:abnavaluen} 
	\end{figure}
\end{example}

\subparagraph*{Background}
Early works on \emph{dynamic flows} consider flows in directed networks where edges have fixed capacities as well as transit times, and where associated costs are minimised
\cite{FF1958,Aronson1989}.

Closest to our setting, Akrida et al.~\cite{AKRIDA201946} compute maximal flows through temporal networks \cite{KKK2000,KKK2002}, where the edge capacities are a function of (discrete and bounded) time, i.e., a fixed capacity word.
They devise a polynomial-time algorithm to compute the maximal \emph{temporal} flow under such constraints, and prove a temporal version of the max-flow min-cut theorem.

In contrast to these works,
in the \sfp\ neither the time horizon nor the capacities at given times are fixed.
Instead, akin to planning or scheduling problems, one can choose the capacity word to maximise the flow through the network.
Sequential flows were introduced in \cite{DBLP:conf/fossacs/ColcombetFO20} in the context of distributed computing.
They have applications for controlling populations of Markov Decision Processes \cite{ColcombetFO21,GMTpop}
and Logics: it can be observed that the \emph{commutative lossy tiling problem} defined and shown decidable in~\cite{BlumensathCP16} is inter-reducible with the \sfp.
Colcombet et al.~\cite{ColcombetFO21} showed that it is \PSPACE-hard, and decidable in exponential space
whether the optimal sequential flow is unbounded\footnote{This is called the ``simple sequential flow problem'' in \cite{ColcombetFO21}.}.
The upper bound was achieved by an exponential reduction to the unboundedness problem of distance automata, for which a PSPACE upper bound is known \cite{DBLP:journals/ita/Kirsten05}.

\subparagraph*{Contributions}
We provide a new, simple and optimal solution for the \sfp.
We show how to compute the precise maximal sequential flow values in polynomial space
(Theorem~\ref{theo:quantitativealgorithm}),
thus matching the 
 lower bound of \cite{ColcombetFO21}
with an upper bound for a much more general problem.
Our technique is adapted to further generalisations:
We derive the same upper bounds for the \sfp\ where sequential flows must be witnessed by capacity words from a given regular language (Theorem~\ref{theo:regularSFP}). This directly generalises the classical \mfp\ as well as the setting in Akrida et al.~\cite{AKRIDA201946}. 
We also show how to solve a generalisation 
called the \msfp{},
where %
the flow should be routed equally 
along a set of given edges (Theorem~\ref{thm:msfp}).

Our contributions are based on new algebraic contributions of independent interest:
In particular, we provide a factorization technique for general finite semigroups,
where we show the existence of small \emph{summaries}  (Theorem~\ref{thm:summary}).
We also show the existence of \emph{$\sharp$-summaries} of polynomial height for elements of the \emph{flow semigroup}, a particular stabilization monoid~\cite{Colcombet13} used to witness unbounded sequential flows.
This allows to
obtain optimal upper bounds on the value of a finite solution to the \sfp\ (Theorem~\ref{thm:boundonflow}).

\section{Sequential Flows}

We first recall the definition of flows and their optimisation problem.

\subparagraph*{Flows}
Fix a finite set of vertices $\vertices$ and two distinct vertices
$v_s$ and $v_t$ referred to as 
\emph{source} and \emph{target}, respectively.
A \emph{flow}
$f \in \RR^{\vertices\x\vertices}$ is a mapping of edges to non-negative reals
which satisfies capacity-- and flow conservation constraints as follows.

\medskip
A \emph{capacity constraint}
$a \in (\N\cup\{\omega\})^{\vertices\x\vertices}$
assigns to each edge a capacity, respecting that
all incoming edges to the source and from the target have capacity $0$.
\begin{align}
\forall (v,v') \in \vertices^2, (v' = v_s \lor v = v_t) \implies a(v,v') = 0\label{inout}
\end{align}
A flow $f$ satisfies the capacity constraint $a$
if %
\begin{align}\label{capcon}
\forall v,v'\in\vertices, \quad f(v,v') \le a(v,v')
\end{align}
It satisfies the \emph{flow conservation} constraints if
\begin{align}\label{flowcon}
\forall v \in \vertices\setminus\{v_s,v_t\},\quad &\outt(f)(v) = \inn(f)(v)
\end{align}
where
$\outt(f)(v) = \sum_{v' \in \vertices}f(v,v')$
and
$\inn(f)(v) = \sum_{v'\in \vertices} f(v',v)$.

\medskip
Intuitively, a flow determines rate of goods
flowing along each of the edges, %
and the capacity of an edge $(v,v')$ is a predetermined bound on the admissible rate that can flow from $v$ to $v'$.
A capacity of $a(v,v') = 0$ means that nothing at all can flow,
and a capacity of $a(v,v')=\omega$ means that an arbitrary finite amount  
can flow.

\medskip
The \emph{value} of a flow $f$ is 
$\abs{f} = \outt(f)(v_s)$, the cumulative flow out of the source vertex.
The \mfp\ asks to compute the maximal value of any flow.

 \probbox{
\AP\intro*\mfp\\
Given a capacity constraint $a \in (\N\cup\{\omega\})^{\vertices\x\vertices}$.

Maximise $\abs{f}$
under constraints in equations \eqref{capcon} and \eqref{flowcon}.
}

Due to the presence of edges with unbounded capacities and because every flow must have a finite value, there may not exist flows of maximal value.

\subparagraph*{Sequential flows}
The \sfp\ is a dynamic variant of this setting in which the maximiser gets to pick a fresh capacity constraint at any unit time.
A sequential flow still represents the rate of goods flowing along the edges, but both capacity constraints and flow conservation dynamically reflect maximiser's momentary choice of capacities.

\AP Assume a finite set
$\Capas \subseteq \left(\NN\cup\{\omega\}\right)^{V\times V}$
of capacity constraints.
A sequence $f=f_1f_2\ldots f_\ell\in \left({\RR^{\vertices\x\vertices}}\right)^*$ is a ""sequential flow"" if for all $0<i\le \ell$, both
\begin{align}
	\tag{\ref{capcon}$'$}\label{scapcon}
\exists a_i\in\Capas, 
\forall v,v'\in\vertices,
\quad 
f_i(v,v') &\le a_i(v,v')\quad\text{and}\\
\tag{\ref{flowcon}$'$}\label{sflowcon}
\forall v \in \vertices\setminus\{v_s,v_t\},
\quad
\inn(f_i)(v)&= \outt(f_{i+1})(v).
\end{align}
\AP A "sequential flow" $f=f_1f_2\ldots f_\ell$ dictates at least one \intro{capacity word} $w=a_1a_2\ldots a_\ell\in\Capas^*$
that witnesses $f$ satisfying the sequential capacity conditions \eqref{scapcon}.
We will refer to $f$ as a "sequential flow" \emph{over} capacity word $w$.
\AP The ""value"" of $f$ is
\[
\abs{f} = \inn(f_\ell)(v_t) = \outt(f_1)(v_s),
\]
the input flow to the target at the latest time $\ell$
and the output of the source at time $1$.
We want to optimise the supremum value of any "sequential flow".

\medskip
\probbox{
\AP  \intro*\sfp\\
Given a finite set $\Capas\subseteq (\N\cup\{\omega\})^{\vertices\x\vertices}$ of capacity constraints.

Determine
the ""optimal sequential flow""
$\optseqflow = \sup\left\{ \abs{f} : f \text{ is a "sequential flow"}\right\}$.
}

\medskip
Note that the \sfp\ is \emph{not} a linear program because $\ell$ is not bounded,
and hence the number of constraints  \eqref{scapcon} and \eqref{sflowcon}
is not bounded a priori.

There is a natural connection between the \sfp\ and the \mfp.
Every capacity word $a_1\ldots a_\ell$
defines an instance of \mfp\ in 
the corresponding pipeline, with source $(v_s,0)$
and target $(v_t,\ell)$.
This instance has a maximal flow,
and $\optseqflow$ is the supremum of those maximal flows
over all capacity words $a_1\ldots a_\ell$.
However, 
there is no simple algorithmic reduction between these two optimisation problems. %
Indeed, 
for an instance of the \sfp\ with capacities $\Capas=\{c\}$,
the "optimal sequential flow" "value"
can be the same, strictly larger, or strictly greater than
the maximal flow value if $c$ is interpreted as an instance of the classical \mfp.

\begin{example}
\label{ex:figurec}
Let $\vertices = \{v_1,v_2,v_3,v_4\}$
and capacities $c,d$, and $e$ %
as depicted in
Figure~\ref{fig:figurec}.
If $c$ is considered as an instance of \textsc{max flow} then only $1$ unit of flow can be transported from source $v_s=v_1$ to target $v_t=v_4$, using the edge $(v_2,v_3)$ at its maximal capacity. 
The values of maximal flows through capacity words $c,cc,ccc,cccc$ are $0,0,1,2$, respectively, and the max flow through any capacity word of the form $c^n, n>4$ remains $2$.
The "optimal sequential flow" given set of capacities $\Capas=\{c\}$ is therefore $2$.

Consider now only capacity $d$.
A flow with maximal value $\abs{f}=2$ 
from the source $v_s=v_1$ to the target $v_t=v_4$
is %
$f(v_1,v_3)=f(v_2,v_4)=f(v_1,v_2)=f(v_3,v_4)=1$.
Similarly, the "sequential flow" $ff$ for capacity word $dd\in\Capas^*$
has value $\abs{ff}=2$.
However, every capacity word $d^n$ of length $n\ge 3$
has maximal flow value $0$.
The "optimal sequential flow" given set of capacities $\Capas=\{d\}$ is therefore $2$.

Consider now only capacity $e$.
The maximal value of a flow from $v_1$ to $v_4$ is $2$ whereas the "optimal sequential flow" is only $1$,
because for any $n\ge0$ there is at most one path of length $n$, and the minimal capacity along these paths is $1$.
The "optimal sequential flow" given set of capacities $\Capas=\{e\}$ is therefore $1$.

Finally, to demonstrate that combining different capacity letters may be required for the "optimal sequential flow",
consider the set of capacities $\Capas=\{c,e\}$. The "optimal sequential flow" "value" $\om$ can be witnessed by a single "capacity word"
$ec$ and a family of "sequential flows"
$f_n=f_{n,1}f_{n,1}$ with $f_{n,1}=f_{n,2}=n$ of value $\abs{f_n}=n$.
\end{example}

\begin{figure}[t]
\begin{subfigure}{0.3\columnwidth}
    \centering
    \vspace*{\fill}
    \begin{tikzpicture}[node distance= 1cm and 2cm]

\node[state] (s) at (0,0) {$v_1$};
\node[state,right of=s] (v) {$v_2$};
\node[state,below of=s] (u) {$v_3$};
\node[state,right of=u] (t) {$v_4$};

\path[use as bounding box] (s) rectangle (t);
\draw[extracolour3,->] (s) edge node[extracolour3] {$\om$} (v);
\draw[extracolour3,->] (v) edge[loop right] node[extracolour3] {$1$} (v);
\draw[extracolour3,->] (v) edge node[extracolour3, above] {$1$} (u);
\draw[extracolour3,->] (u) edge[loop left] node[extracolour3] {$1$} (u);
\draw[extracolour3,->] (u) edge node[extracolour3] {$\om$} (t);

\end{tikzpicture}
    \vspace*{\fill}
    \caption*{\centering capacity $c$}
    \label{fig:figure1-c}
    \label{fig:ex2-c}
  \end{subfigure}
\hfill
\begin{subfigure}{0.3\columnwidth}
    \centering
    \vspace*{\fill}
    \begin{tikzpicture}[node distance= 1cm and 2cm]

\node[state] (s) at (0,0) {$v_1$};
\node[state,right of=s] (v) {$v_2$};
\node[state,below of=s] (u) {$v_3$};
\node[state,right of=u] (t) {$v_4$};

\path[use as bounding box] (s) rectangle (t);
\draw[extracolour1,->] (s) edge node[extracolour1] {$\om$} (v);
\draw[extracolour1,->] (v) edge node[extracolour1] {$1$} (t);
\draw[extracolour1,->] (s) edge node[extracolour1] {$1$} (u);
\draw[extracolour1,->] (u) edge node[extracolour1] {$\om$} (t);

\end{tikzpicture}
    \vspace*{\fill}
    \caption*{\centering capacity $d$}
    \label{fig:figure1-d}
    \label{fig:ex2-d}
  \end{subfigure}
\hfill
\begin{subfigure}{0.3\columnwidth}
    \centering
    \vspace*{\fill}
    \begin{tikzpicture}[node distance= 1cm and 2cm]
\node[state] (s) at (0,0) {$v_1$};
\node[state,right of=s] (v) {$v_2$};
\node[state,below of=s] (u) {$v_3$};
\node[state,right of=u] (t) {$v_4$};

\path[use as bounding box] (s) rectangle (t);
\draw[extracolour2,->] (s) edge node[extracolour2] {$1$} (v);
\draw[extracolour2,->] (s) edge node[extracolour2,swap] {$\om$} (u);
\draw[extracolour2,->] (u) edge node[extracolour2] {$1$} (v);
\draw[extracolour2,->] (v) edge node[extracolour2] {$2$} (t);
\end{tikzpicture}
    \vspace*{\fill}
    \caption*{\centering capacity $e$}
    \label{fig:figure1-e}
    \label{fig:ex2-e}
  \end{subfigure}
  \newline
  \newline
  \begin{subfigure}{0.3\columnwidth}
    \centering
	\begin{tikzpicture}[scale=0.4]
	  \flowwithletter[extracolour3]{$c$}{4}{1-2/$\om$,2-2/$1$,2-3/$1$,3-3/$1$,3-4/$\om$}
	  \flowwithletter[extracolour3,xshift=1.5cm]{$c$}{4}{1-2/$\om$,2-2/$1$,2-3/$1$,3-3/$1$,3-4/$\om$}
	  \flowwithletter[extracolour3,xshift=3cm]{$c$}{4}{1-2/$\om$,2-2/$1$,2-3/$1$,3-3/$1$,3-4/$\om$}
	  \flowwithletter[extracolour3,xshift=4.5cm]{$c$}{4}{1-2/$\om$,2-2/$1$,2-3/$1$,3-3/$1$,3-4/$\om$}
	\end{tikzpicture}
      \caption*{\centering pipeline  $cccc$}%
  \end{subfigure}
  \hfill
  \begin{subfigure}{0.3\columnwidth}
    \centering
	\begin{tikzpicture}[scale=0.4]
	  \flowwithletter[extracolour3]{$c$}{4}{1-2/$2$}
	  \flowwithletter[extracolour3,xshift=1.5cm]{$c$}{4}{2-2/$1$,2-3/$1$}
	  \flowwithletter[extracolour3,xshift=3cm]{$c$}{4}{2-3/$1$,3-3/$1$}
	  \flowwithletter[extracolour3,xshift=4.5cm]{$c$}{4}{3-4/$2$}
	\end{tikzpicture}
      \caption*{\centering maxflow for $cccc$}%
  \end{subfigure}
  \hfill
  \begin{subfigure}{0.3\columnwidth}
    \centering
	\begin{tikzpicture}[scale=0.4]
	  \flowwithletter[extracolour2]{$e$}{4}{1-3/$\om$}
	  \flowwithletter[extracolour3,xshift=1.5cm]{$c$}{4}{3-4/$\om$}
	\end{tikzpicture}
      \caption*{\centering pipeline $ec$}%
  \end{subfigure}
  \caption{\label{fig:figurec}
      Capacities $c,d,e$ from \cref{ex:figurec}, pipeline $cccc$ and its maximal flow, and pipeline $ec$. %
  }
\label{fig:ex1-maxflows}
\label{fig:ex1-pipes}
\end{figure}
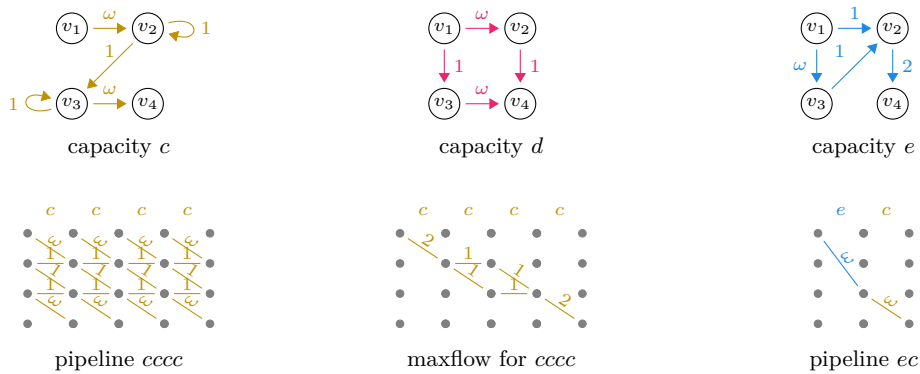

As established in both examples so far, there may not exist sequential flows of maximal value because there are in fact sequential flows of arbitrarily high values.
In that case we call the "optimal sequential flow" \emph{unbounded} and write $\optseqflow = \om$. 
This can be witnessed in two ways: either by a single pipeline such as at the end of \cref{ex:figurec}, or,
as in Example~\ref{ex:intro},
by a family of pipelines of growing length and maximal flow value.
In Example~\ref{ex:intro}, the "optimal sequential flow" "value" of $\om$ cannot be witnessed by any finite pipeline and is instead witnessed by a family $ab^na$ of words that 
iterate the capacity constraint $b$ arbitrarily, but finitely often.
Our final example demonstrates that such witnesses may require more complex nested iterations.

\begin{example}\label{ex:nested}
	Take sets $\vertices=\{v_1,v_2,v_3,v_4,v_5\}$
	and $\Capas=\{a,b,c\}$ of vertices and capacity constraints
	as depicted below, where non-zero capacities of $a,b,c$ are
	shown in {\color{extracolour1}\extracolourOne}, {\color{extracolour2}\extracolourTwo}, and {\color{extracolour3}\extracolourThree}. %
	\begin{center}
	\def\acol{extracolour1}
\def\bcol{extracolour2}
\def\ccol{extracolour3}

\begin{tikzpicture}[node distance= 1.5cm and 2cm]

\node[state] (v1) at (0,0) {$v_1$};
\node[state,right of=v1] (v2) {$v_2$};
\node[state,right of=v2] (v3) {$v_3$};
\node[state,right of=v3] (v4) {$v_4$};
\node[state,right of=v4] (v5) {$v_5$};

\draw[\acol,->] (v1) edge[bend left] node[\acol,above] {$\omega$} (v2);
\draw[\acol,->] (v4) edge[loop right] node[\acol] {$\omega$} (v4);

\draw[\bcol,->] (v2) edge[loop above] node[\bcol] {$\omega$} (v2);
\draw[\bcol,->] (v3) edge[loop above] node[\bcol] {$\omega$} (v3);
\draw[\bcol,->] (v4) edge[loop above] node[\bcol] {$\omega$} (v4);
\draw[\bcol,->] (v2) edge[bend left] node[\bcol,above] {$1$} (v3);

\draw[color=\ccol,->] (v4) edge[loop below] node[color=\ccol] {$\omega$} (v4);
\draw[color=\ccol,->] (v3) edge[bend left] node[color=\ccol,below] {$\omega$} (v1);
\draw[color=\ccol,->] (v3) edge[bend left] node[color=\ccol,above] {$1$} (v4);

\draw[\acol,->] (v4) edge[bend left] node[\acol,above] {$\omega$} (v5);
\end{tikzpicture}
	\end{center}
	The "optimal sequential flow" from source $v_s=v_1$ to target $v_t=v_5$ is $\omega$, yet no finite capacity word witnesses this.
	To witness a "sequential flow" of "value" $n$,
	a "capacity word" must be of the form $(ab^{\ge n}c)^{\ge n} a$.
	The combined capacities for the word $ab^nc$ is
	shown in \cref{fig:ex3} (left).
	This allows a flow of $n$ from $v_1$ to $v_3$ (using $ab^n$);
	then to transfer one unit to $v_3$ (using $c$, which empties $v_3$).
	Iterating this prefix $n$ times allows a flow of $n$ units to $v_4$, at which point all can flow in one step towards the target $v_5$ (via $a$, see the right half of Figure~\ref{fig:ex3}).
\end{example}

\begin{figure}[t]
    \centering
  \setlength{\subfigheight}{3cm} %
  \begin{subfigure}[t][\subfigheight][b]{0.45\textwidth}
    \centering
    \vspace*{\fill}
	{	\begin{tikzpicture}[scale=0.45]
		\node (TL) at (0,1){};
		\node (BR) at (13,-4.5){};
		\path[use as bounding box] (TL) rectangle (BR);

	  \flowlabel[xshift=-0.2cm]{5}{1/$v_1$,2/$v_2$,3/$v_3$,4/$v_4$,5/$v_5$}
	  \flowwithletter[extracolour1]{$a$}{5}{1-2/$\om$,4-4/$\om$,4-5/$\om$}

	  \flowwithletter[extracolour2,xshift=1.5cm]{$b$}{5}{2-2/$\om$,2-3/$1$,3-3/$\om$,4-4/$\om$}

	  \begin{scope}[xshift=-1.5cm]
	  \node[extracolour2] at (6,0.75) {$b^{n-3}$};
	  \path[extracolour2] (4.75,-1) edge[dotted,thick,-] (6.75,-1);
	  \path[extracolour2] (4.75,-1) edge[dotted,thick,-] (6.75,-2);
	  \path[extracolour2] (4.75,-2) edge[dotted,thick,-] (6.75,-2);
	  \path[extracolour2] (4.75,-3) edge[dotted,thick,-] (6.75,-3);
	  \flowwithletter[extracolour2,xshift=7cm]{$b$}{5}{2-2/$\om$,2-3/$1$,3-3/$\om$,4-4/$\om$}
	  \flowwithletter[extracolour3,xshift=8.5cm]{$c$}{5}{3-1/$\om$,3-4/$1$,4-4/$\om$}
	  \end{scope}

	  \node at (9.5,-2) {$\equiv$};
	  
	  \begin{scope}[xshift=11cm]
	  \flowlabel[xshift=-0.2cm]{5}{1/$v_1$,2/$v_2$,3/$v_3$,4/$v_4$,5/$v_5$}
	  \flowwithletter[][yshift=0.1cm]{$ab^nc$}{5}{1-1/$n-1$,1-4/$1$,4-4/$\om$}
	  \end{scope}
	\end{tikzpicture}}
    \vspace*{\fill}
    \caption*{\centering pipeline $ab^nc$ and its combined effect}
  \end{subfigure}
  \hfill
  \begin{subfigure}[t][\subfigheight][b]{0.45\textwidth}
	\centering
    \vspace*{\fill}
    {	\begin{tikzpicture}[scale=0.45]
		\node (TL) at (0,1){};
		\node (BR) at (12,-4.5){};
		\path[use as bounding box] (TL) rectangle (BR);

	  \flowlabel[xshift=-0.2cm]{5}{1/$v_1$,2/$v_2$,3/$v_3$,4/$v_4$,5/$v_5$}
	  \flowwithletter[][yshift=0.1cm]{$ab^nc$}{5}{1-1/$n-1$,1-4/$1$,4-4/$\om$}
	  \flowwithletter[xshift=1.5cm][yshift=0.1cm]{$ab^nc$}{5}{1-1/$n-1$,1-4/$1$,4-4/$\om$}
	\begin{scope}[xshift=3.0cm]
	  \path[extracolour2] (0.25,0) edge[dotted,thick,-] (1.75,0);
	  \path[extracolour2] (0.25,0) edge[dotted,thick,-] (1.75,-3);
	  \path[extracolour2] (0.25,-3) edge[dotted,thick,-] (1.75,-3);
	\end{scope}
	\begin{scope}[xshift=5.0cm]
		\flowwithletter[][yshift=0.1cm]{$ab^nc$}{5}{1-1/$n-1$,1-4/$1$,4-4/$\om$}
	  \flowwithletter[extracolour1,xshift=1.5cm][yshift=0.1cm]{$a$}{5}{1-2/$\om$,4-4/$\om$,4-5/$\om$}
	\end{scope}
	  \node at (9,-2) {$\equiv$};
	  
	  \begin{scope}[xshift=10.5cm]
	  \flowlabel[xshift=-0.2cm]{5}{1/$v_1$,2/$v_2$,3/$v_3$,4/$v_4$,5/$v_5$}
	  \flowwithletter[][yshift=0.1cm]{$(ab^nc)^{n}$}{5}{1-1/$n-1$,1-4/$n$,4-4/$\om$}
	  \flowwithletter[extracolour1,xshift=1.5cm][yshift=0.1cm]{$a$}{5}{1-2/$\om$,4-4/$\om$,4-5/$\om$}
	  \end{scope}

	\end{tikzpicture}}
    \vspace*{\fill}
    \caption*{\centering pipeline for $(ab^{n}c)^{n} a$}
    \label{fig:ex3-pipe-2}
  \end{subfigure}
  \caption{\label{fig:ex3}The pipelines from Example~\ref{ex:nested}.
      The pipeline for $ab^nc$ and its shortened representation (seen on the left)
      is iterated another $n$ times in the pipeline for $(ab^nc)^na$ (seen on the right).
  }
\end{figure}
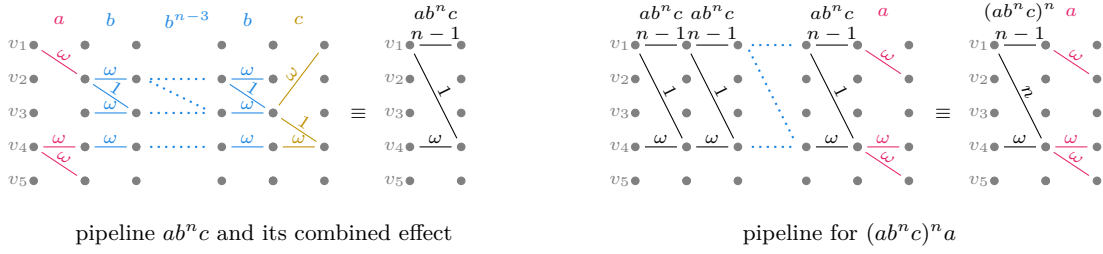

\section{Solving the \sfp}\label{sec:seqflow}
We present a solution to the \sfp\ in two stages. 
The first stage, described in Section~\ref{subsec:qualitative}, is qualitative: we determine whether the instance is unbounded, i.e., whether there exists sequential flows of arbitrarily high value. The key for that is to abstract the exact computation of values by means of a finite algebraic structure called the \emph{flow semigroup}. 
The elements of this semigroup are enumerated using Algorithm~\ref{algoquali}, which searches for a witness of unboundedness. 
The second stage is quantitative: it is performed by Algorithm~\ref{algo},
which computes the maximal value of sequential flows, assuming they are bounded.
Both stages can be carried out in polynomial space. A key component for this upper bound in the second stage is the proof that when sequential flows are bounded the supremum is at most exponential in the number of vertices.

\label{subsec:qualitative}

For the qualitative stage, we will
abstract the exact values of 
capacity constraints and only consider whether those values are $0$, finite or $\omega$.

\AP We make use of the ""maxmin semiring"" 
\[
\intro*\mmsm = (\set{0,1,\omega},\max,\min) \text{ with } 0 < 1 < \omega
\]	

There is a natural structure of semigroup on $\mmsm^{V \times V}$,
using the usual matrix product over this semiring. For $x,y \in \mmsm^{V \times V}$,
and $v,v' \in V$,
\[
 (x\cdot y)(v,v')
=
\max_{v''\in V} \left(\min\left(x\left(v,v''\right),y\left(v'',v'\right)\right)\right)\enspace.
\]
Every capacity constraint $a\in \NN^{V \times V}$ is naturally abstracted
as a matrix $x_a \in \mmsm^{V \times V}$
by losing precision: $0$ and $\omega$ are preserved while 
finite positive numbers are mapped to $1$.

\begin{example}\label{ex:quali}
In Example~\ref{ex:intro}
there are two capacity constraints $a$ and $b$.
Their abstractions, as well as the product thereof, are as follows:
\[
x_a = \begin{pmatrix}
0 & \om & 0 & 0  \\
0 & 0 &0 &0 \\
0 & \om &0 &\om \\
0 & 0 & 0 &0 
\end{pmatrix}
\enspace
\qquad
x_b = 
\begin{pmatrix}
0 & 0 & 0 & 0  \\
0 & \om &1 &0 \\
0 & 0 & \om &0 \\
0 & 0 & 0 &0 
\end{pmatrix}
\qquad
x_a\cdot x_b
=
\begin{pmatrix}
0 & \om & 1 & 0  \\
0 & 0 &0 &0 \\
0 & \om & 1 &0 \\
0 & 0 & 0 &0 
\end{pmatrix}
\enspace
\]

Notice that $a = x_a$ and $b = x_b$ coincide with their abstractions %
since all finite coordinates are equal to $0$ or $1$.
\end{example}

The matrix product allows to keep track of which paths have finite or $\omega$-capacity. For every $n\geq 1$, the product $x_ax_b^nx_a$ can be computed easily: since $x_b$ is idempotent, i.e. $x_b^2 = x_b$, we have
\[
x_a\cdot x_b^n \cdot x_a = x_a \cdot x_b \cdot x_a =
\begin{pmatrix}
0 & 1 & 0 & 1  \\
0 & 0 &0 &0 \\
0 & 1& 0 &1 \\
0 & 0 & 0 &0 
\end{pmatrix}
\]
This simple computation tells us that starting from the source $v_s$ (the first line), any "sequential flow" through the pipeline $x_ax_b^nx_a$ can carry only a finite flow to the target $v_t$ (the last column), because the top right coefficient is $1$.
However, this finite representation of the pipeline $x_ax_b^nx_a$ misses an important point: although the flow from $v_s$ to $v_t$ is finite for every $n>0$, it is actually unbounded and grows with $n$ (cf. Figure~\ref{fig:abnavaluen}). 

In order to take this phenomenon into account, we introduce an extra operation on idempotent elements, i.e. the elements $e \in \mmsm^{V\times V}$ such that $e=e^2$. This operation computes a new idempotent element $e^\sharp\in \mmsm^{V\times V}$. Intuitively, the element $e^\sharp$ is an abstraction of the sequence of pipelines $(e^n)_{n \in \NN}$ and should be thought of as ``using $e$ many times''. 

Before we proceed to define this operation, let us observe that
the possible effect that iterating idempotents can have on the maximal flow between any two vertices.

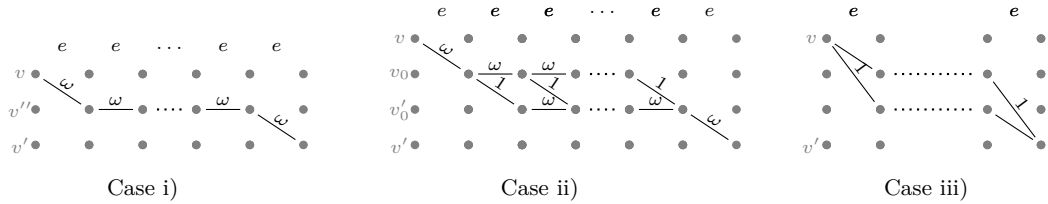
\begin{figure}[h]
	\begin{subfigure}[t]{0.27\textwidth}
	\begin{tikzpicture}[scale=0.47]
	  \flowlabel[xshift=-0.2cm]{6}{1/$v$,2/$v''$,3/$v'$}
	  \flowwithletter{$e$}{3}{1-2/$\om$}
	  \flowwithletter[xshift=1.5cm]{$e$}{3}{2-2/$\om$}
	  
	  	  \flowwithletter[xshift=3cm]{$\ldots$}{3}{}

	  	\begin{scope}[xshift=3cm]
	  \path (0.25,-1) edge[dotted,thick,-] (1.25,-1);
	\end{scope}
	  \flowwithletter[xshift=4.5cm]{$e$}{3}{2-2/$\om$}
	  \flowwithletter[xshift=6cm]{$e$}{3}{2-3/$\om$}
	\end{tikzpicture}
\caption*{\centering Case i)}%
\end{subfigure}
\hspace{1cm}
\begin{subfigure}[t]{0.3\textwidth}
	\begin{tikzpicture}[scale=0.47]
	  \flowlabel[xshift=-0.2cm]{6}{1/$v$,2/$v_0$,3/$v_0'$,4/$v'$}
	  \flowwithletter{$e$}{4}{1-2/$\om$}
	  
	  \flowwithletter[xshift=1.5cm]{$e$}{4}{2-2/$\om$}
	  \flowwithletter[xshift=1.5cm]{$e$}{4}{2-3/$1$}
	  
	  \flowwithletter[xshift=3cm]{$e$}{4}{2-2/$\om$}
	  \flowwithletter[xshift=3cm]{$e$}{4}{3-3/$\om$}
	  \flowwithletter[xshift=3cm]{$e$}{4}{2-3/$1$}

	  \flowwithletter[xshift=4.5cm]{$\ldots$}{4}{}

	  	\begin{scope}[xshift=4.5cm]
	  \path (0.25,-1) edge[dotted,thick,-] (1.25,-1);
	  \path (0.25,-2) edge[dotted,thick,-] (1.25,-2);
	\end{scope}

	  \flowwithletter[xshift=6cm]{$e$}{4}{3-3/$\om$}
	  \flowwithletter[xshift=6cm]{$e$}{4}{2-3/$1$}

	  \flowwithletter[xshift=7.5cm]{$e$}{4}{3-4/$\om$}
	
	\end{tikzpicture}
\caption*{\centering Case ii)}%
\end{subfigure}
\hspace{1cm}
\begin{subfigure}[t]{0.25\textwidth}
	\begin{tikzpicture}[scale=0.47]
	  \flowlabel[xshift=-0.2cm]{6}{1/$v$,2/$$,2/$$,4/$v'$}
	  \flowwithletter{$e$}{4}{1-2/$$}
	  \flowwithletter{$e$}{4}{1-3/$1$}

	  	\begin{scope}[xshift=1.5cm]
	  \path (0.25,-1) edge[dotted,thick,-] (2.75,-1);
	  \path (0.25,-2) edge[dotted,thick,-] (2.75,-2);
	\end{scope}
	  \flowwithletter[xshift=4.5cm]{$e$}{4}{2-4/$1$}
	  \flowwithletter[xshift=4.5cm]{$e$}{4}{3-4/$$}
	\end{tikzpicture}
\caption*{\centering Case iii) }%
\end{subfigure}

\hfill

	\caption{\label{fig:idempotent}The figure illustrates Lemma~\ref{lem:idempotent}, which classifies
	the three possible long-term behaviours of an edge $(v,v')$ of an idempotent $e$, when $e(v,v')\neq 0$. Case i) is $e(v,v') =\om$ and case ii) and iii) occur when $e(v,v') =1$.}
\end{figure}

\begin{lemma}[Flow-carrying Edges of idempotent elements]\label{lem:idempotent}
Let $e\in \mmsm^{V\times V}$ such that $e=e^2$,
and $v,v' \in V$ such that $e(v,v')>0$.
For $n\geq 1$, let $K_n$ denote the optimal flow value from $v$ to $v'$ in the pipeline $e^n$.
Exactly one of the following holds.
\begin{enumerate}
\item[i)]
$K_n=\omega$ for all $n\geq 1$.
This holds iff $e(v,v')=\omega$ and there exists $v''\in V$ such that 
\[
\omega=e(v,v'')=e(v'',v'')=e(v'',v').
\]
\item[ii)]
$n-2 \leq K_n \leq n|V|$ for all $n\geq 1$.
This holds iff $e(v,v')=1$ and there exists $v_0,v_0'\in V$ such that
\[
e(v_0,v_0')=1 \text{ and } \omega = e(v,v_0)=e(v_0,v_0)=e(v'_0,v'_0)=e(v'_0,v').
\]
\item[iii)]
$1 \leq K_n\leq 2 |V|$ for all $n\geq 1$.
This holds iff $e(v,v')=1$ and for all $v_0,v_0'\in V$,
\[
e(v_0,v'_0)\geq 1 \implies (
e(v,v_0)\leq 1 \text{ or } e(v'_0,v')\leq 1).
\]
\end{enumerate}
\end{lemma}

The following definition of iterations of idempotents explicitly distinguishes the three cases of Lemma~\ref{lem:idempotent} to summarise an ``ever growing'' finite maxflow (case ii) as a new $\omega$.

\begin{definition}[iteration of an idempotent]\label{def:idempotent}
	Let $e=e^2$ be an idempotent element of  $\F$.
	A pair $(v,v')\in V^2$ such that $e(v,v')=1$ is called ""unstable@@pair"" 
iff there exists $v_0,v'_0\in V$ such that $e(v,v_0)=\omega$, $e(v_0,v'_0)=1$
and $e(v'_0,v')=\omega$,
and ""stable@@pair"" 
otherwise.

	Then the \emph{iteration} of $e$, denoted $e^\sharp$, is defined by
	\[
	e^\sharp(v,v')=
	\begin{cases}
		e(v,v') & \text{ if $e(v,v')\in\{0,\omega\}$}             \\
		1      & \text{ if $e(v,v')=1$ and $(v,v')$ is "stable@@pair" in $e$}    \\
		\omega & \text{ if $e(v,v')=1$ and $(v,v')$ is "unstable@@pair" in $e$.}
	\end{cases}
	\]
	An idempotent $e$ is called ""unstable@@idempotent"" if it has an "unstable@@pair" pair, and ""stable@@idempotent"" otherwise.
\end{definition}

Note that if an edge is "unstable@@pair" then it does not 
satisfy condition iii) of Lemma~\ref{lem:idempotent},
 thus it satisfies condition ii) of the same Lemma.
  Note also that $e$ is "unstable@@idempotent" if and only if $e \neq e^\sharp$.
  We also make the following observation, proven in Appendix~\ref{app:sharpsharp}.

\begin{lemma}\label{lem:sharpsharp}
The iteration of an idempotent $e$ is stable,
i.e.,
$(e^\sharp)^\sharp = e^\sharp$.
\end{lemma}

Our main algebraic tool is the "flow semigroup",
a finite structure obtained by application of the product and iteration
to the abstract capacities until saturation.
This "flow semigroup" is in fact an example of a \emph{stabilisation monoid} (upon addition of a neutral element), which are ordered monoids with a stabilisation operator 
\cite{Colcombet13}.

\begin{definition}[flow semigroup]\label{def:flowsemigroup}
 The ""flow semigroup"", denoted $\intro*\F$,
	is the smallest subset of  $\mmsm^{V \times V}$
	which contains all abstracted capacity constraints $\{ x_a \mid a \in \Capas\}$
	and is closed under the matrix product and the iteration operation.
\end{definition}

\begin{example}\label{ex:flowsemigroup}
Continue with Example \ref{ex:quali}.
The "flow semigroup" $\F$ contains $x_a$ and $x_b$
and since $x_b^2 = x_b$, it also contains $x_b^\sharp$. The only capacity-$1$ edge %
in $x_b$ is unstable
since $x_b(v_2,v_2)=\omega, x_b(v_2,v_3)=1$ and $x_b(v_3,v_3)=\omega$.
Therefore, $\F$ contains the iteration

\bigskip
\begin{minipage}{0.5\columnwidth}
	\centering
$
x_b^\sharp = 
\begin{pmatrix}
0 & 0 & 0 & 0  \\
0 & \om &\om &0 \\
0 & 0 & \om &0 \\
0 & 0 & 0 &0 
\end{pmatrix}
$
\end{minipage}%
\begin{minipage}{0.5\columnwidth}
	\centering
\begin{tikzpicture}[node distance= 1cm and 2cm]

\node[state] (s) at (0,0) {$v_1$};
\node[state,right of=s] (u) {$v_2$};
\node[state,below of=s] (v) {$v_3$};
\node[state,right of=v] (t) {$v_4$};

\draw[->] (v) edge[loop left] node {$\omega$} (v);
\draw[->] (u) edge[loop right] node {$\omega$} (u);
\draw[->] (u) edge node[above] {$\omega$} (v);
\end{tikzpicture}
\end{minipage}

\bigskip

Intuitively, $b$ allows only one unit of flow from $v_2$ to $v_3$, but arbitrarily much flow can remain both in $v_2$ and $v_3$.
If this is iterated then the sum of those single units of flows 
from $v_2$ to $v_3$ grows, which is represented by a new capacity $\omega$ on this edge
in the capacity constraint $b^\sharp$.

Finally, $\F$ contains the product %

\bigskip

\begin{minipage}{0.5\columnwidth}
	\centering
$
x_a x_b^\sharp x_a = 
\begin{pmatrix}
0 & \om &  0 & \om  \\
0 & 0 &0 &0 \\
0 & \om& 0 &\om \\
0 & 0 & 0 &0 
\end{pmatrix}
$
\end{minipage}%
\begin{minipage}{0.5\columnwidth}
	\centering
\begin{tikzpicture}[node distance= 1cm and 2cm]

\node[state] (s) at (0,0) {$v_1$};
\node[state,right of=s] (u) {$v_2$};
\node[state,below of=s] (v) {$v_3$};
\node[state,right of=v] (t) {$v_4$};

\draw[->] (s) edge[] node {$\omega$} (u);
\draw[->] (s) edge[left] node {$\omega$} (t);
\draw[->] (v) edge[right] node {$\omega$} (u);
\draw[->] (v) edge[] node {$\omega$} (t);
\end{tikzpicture}
\end{minipage}

\bigskip

The top right coordinate 
in $x_a x_b^\sharp x_a$,
which corresponds to the edge $(v_1,v_4)$ from the source to the target,
is $\om$.
This suggests that an arbitrary amount of flow can be transported from $v_s=v_1$ to $v_4=v_t$. Such an element is called an \emph{unboundedness witness}.
\end{example}

\begin{definition}[Unboundedness witness]\label{def:unboundednesswitness}
An ""unboundedness witness"" is an element of $x \in \F$
such that
\[
x(v_s,v_t) = \omega\enspace.
\]
\end{definition}

The existence of such elements  
in $\F$ is a sufficient and necessary condition
for the existence of sequential flows
carrying an arbitrary large amount of flow
from the source to the target.

\begin{theorem}[characterization of the unboundedness case]\label{thm:characterisation}
An instance of the \sfp\ is unbounded
if, and only if, there exists an "unboundedness witness" in 
the corresponding flow semigroup.
\end{theorem}

We will now present the structure of the proof of this theorem. 

To begin with, we observe that since our finite capacities have integer values the integral flow theorem~\cite{FF1958} guarantees that every sequence of capacity constraints has a flow with integer coefficients and maximum value.
This allows us to work with \emph{token flows},
in which an explicit set of named tokens is fixed,
and the trajectory of every token is precisely described.

\newcommand{\Tokens}{\T}
\newcommand{\tok}{\tau}

\begin{definition}[Token flows]\label{def:tokenflow}
Fix a finite set $\Tokens$ of tokens.
A ""token flow"" of $\Tokens$ of length $k$ over a "capacity word" $w$ is a mapping $\dsf \in V^{\Tokens\times (0\ldots k)}$
which describes the position of every token at every date, and which satisfies capacity constraints.
Formally,
for every $i \in [1,k]$ and $v,v' \in V$
\[
\abs{\{\tok \in \Tokens \mid \dsf(\tok,i-1)=v \land \dsf(\tok,i) = v' \}} \leq a_i(v,v').
\]
\end{definition}

The outcome of a "token flow" is described by a \emph{global flow},
which accounts for the number of tokens traveling along every pair of states.

\begin{definition}[Global flows]
With every token flow $\dsf$ of length $k$ is associated a ""global flow"" denoted $\intro*\gsf(\dsf) \in \NN^{V \times V}$,
which measures the number of tokens moving between every pair of vertices between the dates $0$ and $k$,
formally defined as:
\[
	\gsf(\dsf)(v,v') = \abs{\{\tok \in \Tokens \mid \dsf(\tok,0)=v \land \dsf(\tok,k) = v' \}}
\] 
for all  $v, v' \in V$.
\end{definition}

The proof of Theorem~\ref{thm:characterisation} is in two steps: Lemma~\ref{lem:sufficientwitness}
establishes that the condition is sufficient 
and Lemma~\ref{lem:necessarywitness} that is it is necessary.

\begin{lemma}[sufficient condition]\label{lem:sufficientwitness}
	If there exists an "unboundedness witness",
	the answer to the \sfp\ is $\omega$.
\end{lemma}

The proof makes use of the notion of paths.

\begin{definition}[Paths]\label{def:pathpipeline}
A ""path@@pipeline"" in a "capacity word" $\capa = a_1\dots a_k \in A^*$ is a sequence of vertices $\pi : \set{0,\dots, k} \to V$ that may be followed by a token, i.e., such that $a_i(\pi(i-1), \pi(i))\geq 1$ for all $i \in \set{1,\ldots, k}$.
\end{definition}

\knowledgenewrobustcmd{\diamondproperty}{\cmdkl{($\diamondsuit$)}\xspace}
\begin{proof}[Proof sketch of Lemma~\ref{lem:sufficientwitness}]
	The full proof is presented in Appendix~\ref{app:sufficient}.
	The key idea is to establish that the following property
	of element $x$ in the flow semigroup is invariant by product and iteration, and satisfied by the abstract capacities.
The property 
\diamondproperty
is satisfied by $x \in \F$
if for every $N\geq 1$,
	there exists a "token flow" $\dsf(x)$ over a  "capacity word" $\capa$
	such that the corresponding global flow $\gsf(\dsf(x))$ satisfies
	\[
	\forall v,v' \in V, x(v,v') = \omega \Rightarrow \gsf(\dsf(x))(v,v') \geq N\enspace,
	\]
	as well as
	\[
	\forall v,v' \in V, x(v,v') = 1 \Rightarrow \text{ there is a "path@@pipeline" in $\capa$ from $v$ to $v'$}.
	\]
	
	The proof is by induction on the "flow semigroup".
	For abstract capacity constraints $x_a$,
	the "capacity word" is simply $a$
	and we have $N$ tokens following each edge $(v,v')$ with $x_a(v,v') = \omega$.
	
	For a product $x \cdot y$
	the token flow $\dsf(x\cdot y)$ for $N$ is obtained by considering two token flows $\dsf(x)$ and $\dsf(y)$ for $|V|^2 N$ tokens, renaming and deleting some of the tokens and concatenating the resulting sequential flows.
	
	The last case is the iteration $f=e^\sharp$.
	We start by showing that we can decompose $e^\sharp$ as a product of $e$ with so-called "simple unstable idempotents", 
	where all non-zero entries are self-loops, except for a single "unstable@@pair" pair. 
	This simple structure is used to craft
	token flows carrying arbitrarily large amount of tokens
	along "unstable@@pair" pairs.
\end{proof}

For the other direction, we show that we can evaluate "capacity words" in the "flow semigroup" so that the flow between pairs of vertices which the resulting element of $\F$ maps to $0$ or $1$ are uniformly bounded.. 

\begin{theorem}\label{thm:boundonflow}
Let $K$ be the largest finite constant appearing in a capacity constraint of $\Capas$.
For every "token flow" $\dsf$,
there exists an element $x \in \F$ such that,
for every $v,v'\in V$,
\begin{align*}
x(v,v') = 0 &\implies \gsf(\dsf)(v,v') = 0 \qquad\text{and} \\
x(v,v') = 1 &\implies \gsf(\dsf)(v,v') \leq \boundV.
\end{align*}
\end{theorem}

This result relies on a subtle and careful analysis of the flow semigroup, performed in Section~\ref{sec:diving}.
The first step is a general result about finite semigroups,
showing that every element of a semigroup $\monoid$
has a finite representation as a binary tree whose height is polynomial in
$\Jlength{\monoid}$, the "regular $\mathcal{J}$-length" of the semigroup (Theorem~\ref{thm:summary}).
The second step is specific to the "flow semigroup" $\F$,
and shows that for $\F$ this parameter is polynomial in $\abs{V}$ (Theorem~\ref{thm:tsef}). We then extend these trees to incorporate the $\sharp$ operator, and prove that the resulting trees still have polynomial height.
That leads to Theorem~\ref{thm:boundonflow}.
By contraposition, a consequence of Theorem~\ref{thm:boundonflow} is the following.

\begin{lemma}[necessary condition]\label{lem:necessarywitness}
If the answer to the \sfp\ is $\omega$ then there is an "unboundedness witness" in $\F$. 
\end{lemma}

\begin{proof}
	Since we have sequential flows of unbounded "values" between $v_s$ and $v_t$, there exists a "capacity word" $\capa$ with a "sequential flow" of value greater than $\boundV$.
	By the integral flow theorem~\cite{FF1958}, we also have a "token flow" $\dsf$ over $\capa$ such that $\gsf(\dsf)(v_s,v_t) > \boundV$. 
	We apply Theorem~\ref{thm:boundonflow} to $d$.
	By case inspection, the only possible value of $x(v,v')$
	is $\omega$, thus $x$ is an "unboundedness witness".
\end{proof}

In order to test unboundedness for the \sfp\
we can compute the entire "flow semigroup" $\F$, starting with the capacity abstractions $\{x_a \mid a \in A\}$ and closing it by product and $\sharp$,
and then check if it contains an "unboundedness witness".
The correctness of 
this algorithm follows directly from
the definition of $\F$, as well as
Lemma~\ref{lem:sufficientwitness} and Lemma~\ref{lem:necessarywitness}.
The resulting algorithm runs in exponential time and space, essentially bounded by the size of $\F$.
This can be improved to polynomial space: Instead of explicitly enumerating elements of the flow group we can enumerate so-called $\sharp$-expressions, which represent elements of $\F$.

\begin{definition}\label{def:sharpexpression}
A ""$\sharp$-expression"" of an element $x\in \F$
is a finite $\F$-labeled ordered tree such that
the root node is labeled by $x$,
and every node $\nu$ is of one of three possible types:
\begin{itemize}
\item either a leaf node labeled by an abstract capacities $x_a,a\in A$;
\item it has a single child labeled by $e$, in which case $e=e^2$ is idempotent and $\nu$ is labeled by $e^\sharp$,
\item it has two children labeled by $x_1$ and $x_2$, in which case 
$\nu$ is labeled by $x_1 \cdot x_2$.
\end{itemize}
\end{definition}

The recursive nature of this definition dictates a recursive algorithm to determine if
a given element $x \in \set{0,1,\om}^{V \times V}$ has a $\sharp$-expression of at most a given height $h$.
By convention, the height of a single leaf node is $0$. Correctness is shown in Appendix~\ref{app:algo-sharp-exp}.

\algnewcommand{\IfThenElse}[3]{%
  \State \algorithmicif\ #1\ \algorithmicthen\ #2\ \algorithmicelse\ #3}

\begin{algorithm}
	\caption{\label{algoquali}
Check if $x$ has a $\sharp$-expression of height at most $h$.}
\begin{algorithmic}[1]
\Function{IsIn-Rec}{x, h}
	\If{$x \in \{x_a \mid a \in A\}$}
	{\Return true}\label{alg-base-1}	
	\EndIf
	\If{h=0}
		{\Return false}\label{alg-base-2}	
	\EndIf	
	\For {every $y, z \in \{0,1,\omega\}^{V \times V}$ with $y\cdot z = x$}
	\label{alg-enum-1}	
			\If{$\textsc{IsIn-Rec}(y, h-1)$ and $\textsc{IsIn-Rec}(z,h-1)$}
			\Return {true}\label{alg-true-1}	
			\EndIf
		\EndFor
		\For {every $e \in \{0,1,\omega\}^{V \times V}$ with $e\cdot e = e$ and $e^\sharp = x$}
	\label{alg-enum-2}	
			\If{$\textsc{IsIn-Rec}(e, h-1)$}
			\Return {true}\label{alg-true-2}	
			\EndIf
		\EndFor
		\State 
		\Return {false}\label{alg-end}	
\EndFunction
	\end{algorithmic}
	\end{algorithm}

\begin{lemma}\label{lem:algoquali}
Given capacities $\Capas$, 
element $x \in \set{0,1,\om}^{V \times V}$, and $h\ge 0$,
Algorithm~\ref{algoquali} returns \emph{true} iff
$x$ has a $\sharp$-expression of height at most $h$.
\end{lemma}

Notice that Algorithm~\ref{algoquali} still runs in (deterministic) exponential time
due to the enumerations in lines \ref{alg-enum-1} and \ref{alg-enum-2}. However, it only requires space polynomial in $\abs{\vertices}$ and $h$ due to the explicit bound on the recursion depth.

The central argument for showing that unboundedness can be tested in polynomial space is a polynomial bound on the number of nested applications of the $\sharp$ operator necessary to produce an unboundedness witness, provided by the following theorem.

	\begin{theorem}\label{cor:sharpexpression}
	Every element of the flow semigroup $\F$ is generated by a 
	"$\sharp$-expression" of height at most  
	$\sharpexpbound$.
	\end{theorem}

We make use of proof techniques used for designing polynomial-space algorithms in other contexts, in particular for checking limitedness of desert automata~\cite{DBLP:journals/ita/Kirsten05} and the value $1$ problem of probabilistic automata~\cite{value1,stamina}.
According to Theorem~\ref{cor:sharpexpression},
every element of $\F$ has a $\sharp$-expression 
of height at most $\sharpexpbound$.
We can therefore check in
polynomial space
if a given element $x$ is in the flow semigroup $\F$
and whether there exist unboundedness witnesses.

\begin{theorem}[Checking unboundedness]\label{theo:qualitativealgorithm}
Checking whether
the "optimal sequential flow" is unbounded
is decidable in polynomial space.
\end{theorem}
\begin{proof}
By Theorem~\ref{cor:sharpexpression}, any positive instance admits an unboundedness witnessed $x$ that has $\sharp$-expressions of height at most $h=\sharpexpbound$.
It therefore suffices to enumerate (in polynomial space, using Algorithm~\ref{algoquali}) all
$\sharp$-expressions of such bounded height and for each check if the represented element $x \in \F$ %
constitutes an unboundedness witness,
i.e., that $x(v_s, v_t) = \om$. 
\end{proof}

If the maximal finite sequential flow value exists,
then we can compute it in polynomial space
thanks to the exponential bound $K$ established in
Theorem~\ref{thm:boundonflow}.
Indeed, if there is a flow of value above $K$ then there are flows of unbounded values.
It therefore suffices to check whether there is a flow of value $K+1$. If so then flows can have unbounded values, if not, then we find the optimal value between $0$ and $K$ by dichotomic search.
Hence we only need to be able to check whether there is a flow of a given, at most exponential, value. This is done in polynomial space via classic graph exploration.
Full details are in Appendix~\ref{app:pspace-algo}.

\begin{theorem}%
\label{theo:quantitativealgorithm}
Given capacities $\Capas$, the optimal sequential flow
can be computed in \PSPACE.
\end{theorem}

\section{Diving into the flow semigroup}\label{sec:diving}

This section is dedicated to the proof of two bounds which are crucial to show that the \sfp\ can be solved in \PSPACE. The first one is Theorem~\ref{cor:sharpexpression}, which gives a polynomial upper-bound on the maximal depth of a $\sharp$-expression generating elements of $\F$. The second one is Theorem~\ref{thm:boundonflow},
which establishes a polynomial upper bound on the "optimal sequential flow", in case it is finite.

The central tool for these proofs are "summaries" and "$\sharp$-summaries",
which provide a bounded-size representation of words of arbitrary length.

\subsection{A general factorization theorem for finite semigroups}

\AP We rely on a form of factorization of words in a finite semigroup, which we call \emph{"summaries"}.
Let $(\monoid, \cdot)$ be a finite semigroup\footnote{We choose to work with semigroups in this paper since they are slightly more general than monoids. Note that~\cite{Jecker21} formulates everything for finite monoids, but the existence of a neutral element is never used in the paper. Every statement from that paper  holds for finite semigroups as well.}.
An element $e \in \monoid$ is called ""idempotent"" if $e \cdot e = e$.
The set of "idempotent" elements of $\monoid$ is denoted $\intro*\idempotents{\monoid}$.
We write $\monoid^{\mathbf{1}}$ for the monoid obtained by extending $\monoid$ with a neutral element $\mathbf{1}$.

We define a morphism $\intro*\evalmorph{\monoid} : \monoid^* \to \monoid$ that evaluates sequences of elements of $\monoid$ by applying the semigroup (product) operation: $\evalmorph{\monoid}(x_0 \dots x_k) = x_0 \cdot x_1 \cdots x_k$.

\AP We recall the Green relations, introduced in~\cite{green}, $\Jgreen, \Lgreen,\Rgreen, \Hgreen$ on $\monoid$, starting with the following partial orders:
\begin{itemize}
	\item $ x \intro*\Jleq y$ if there exist $a,b \in \monoid^{\mathbf{1}}$ such that $x = a y b$
	\item $ x \intro*\Lleq y$ if there exist $a \in \monoid^{\mathbf{1}}$ such that $x = a y$
	\item $ x \intro*\Rleq y$ if there exist $b \in \monoid^{\mathbf{1}}$ such that $x = y b$
	\item $ x \intro*\Hleq y$ if $x \Lleq y$ and $x \Rleq y$	
\end{itemize}
These partial orders can be thought of as reachability relations on $\monoid$, and the corresponding equivalence classes as strongly connected components. 

\AP We use well-established notations for the relations $\intro*\Jgreen, \intro*\Lgreen, \intro*\Rgreen, \intro*\Hgreen$, the equivalence relations induced by those partial orders.
Formally,
for each $\mathcal{X} \in \set{\mathcal{J}, \mathcal{L}, \mathcal{R}, \mathcal{H}}$ we define the relations $\mathcal{X} = \leq_{\mathcal{X}} \cap \geq_{\mathcal{X}}$ and $<_{\mathcal{X}} = \leq_{\mathcal{X}} \setminus \geq_{\mathcal{X}}$. 
Note that $\Jgreen$ is coarser than $\Lgreen$ and $\Rgreen$, themselves coarser than $\Hgreen$.

\AP The ""regular $\mathcal{J}$-length""
\footnote{Called regular $\mathcal{D}$-length in~\cite{Jecker21}. For finite semigroups, $\mathcal{D} = \mathcal{J}$, and we use $\mathcal{J}$ here since it is more common. Note that the paper introduces it with a different definition, but the two are proven equivalent in the extended version~\cite[Appendix B]{JeckerArxiv}.} 
of a semigroup $\monoid$, denoted $\intro*\Jlength{\monoid}$, is defined as 
\[\sup\set{k \in \NN \mid \exists e_1 \Jlt \cdots \Jlt e_k \in \idempotents{\monoid}}.\]

\AP The ""Ramsey function"" of $\monoid$ is the function $R_{\monoid} : \NN \to \NN$ where $\intro*\Ramsey{\monoid}{k}$ is the minimal number $n$ such that every word $w \in \monoid^*$ of length $n$ contains an infix of the form $u_1 \cdots u_k$ with $\phi(u_1) = \cdots = \phi(u_k) = e$ for some $e \in \idempotents{\monoid}$.
The existence of such a number $n$ for all $k$ can be inferred from Ramsey's theorem, but the following theorem gives much more precise bounds.

A core element of the construction is the following result, which guarantees the existence of consecutive idempotent factors in sufficiently long words over $\monoid$. 
Note that the bound is only exponential in the "regular $\mathcal{J}$-length" of the semigroup, not its size.

\begin{theorem}[{\cite[Theorem 1]{Jecker21}}]\label{thm:Ramseybounds}
	For all $k \in \NN$, \[k^{\Jlength{\monoid}} \leq \Ramsey{\monoid}{k} \leq (k|\monoid|^4)^{\Jlength{\monoid}}.\]
\end{theorem}

We can now define the central object of our proofs. 
The theorem gives a way to summarize a word with respect to a finite semigroup. A "summary" abstracts sequences of idempotent infixes by only keeping the first and last ones.
We do so in a way that ensures that the remaining idempotent factors are ``short'', so that the number of letters we keep from the initial word is polynomial in the size of the semigroup and $\Ramsey{\monoid}{3}$.
\begin{definition}
A ""summary"" of a word $w \in \monoid^*$  is a $\monoid\times \monoid^*$-labeled ordered binary tree with three types of nodes:
\begin{itemize}
	\item a leaf has no children and a label $(x,x)$ for some $x \in \monoid$
	
	\item a product node labeled $(x,w)$ has two children labeled $(y_1, u_1)$ and $(y_2,u_2)$ such that $y_1\cdot y_2 = x$ and $u_1 u_2 = w$.
	
	\item an idempotent node labeled $(e,w)$ has two children labeled $(e,u_1)$ and $(e,u_2)$ such that we have
	$w = u_1 w' u_2$, $e \in \idempotents{\monoid}$ and $\evalmorph{\monoid}(u_1) = \evalmorph{\monoid}(u_2) =  \evalmorph{\monoid}(w') = e$.	
\end{itemize}
The root is labeled by $(x,w)$, for some $x \in F$, called the ""result@@fact"" of the "summary".
\end{definition}

\begin{example}\label{ex:shortendfactorisationexample}
Consider the flow semigroup $\F$ for Example~\ref{ex:intro}.
Remember that $a=x_a$ and $b=x_b$: capacities match their abstractions because finite constants are $0$ or $1$.
Let $z,y\in \F$ be the products $y=b \cdot a$ and $z=a \cdot b \cdot a$.
Since $b=b^2$, 
the following tree is a summary of $ab^na$, for any $n>2$.
	
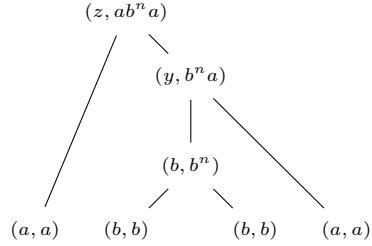
\begin{center}
	\begin{tikzpicture}[node distance= 0.25cm and 0.6cm]

\node (abba) at (0,0) {$(z,ab^na)$};

\node[below right=of abba] (bba) {$(y,b^na)$};
\draw (bba) edge[-][-] (abba);

\node[below=of bba] (bb) {$(b ,b^n)$};
\draw (bb) edge[-] (bba);

\node[below left=of bb] (b1) {$(b,b)$};
\node[below right=of bb] (b2) {$(b,b)$};
\draw (b1) edge[-] (bb);
\draw (b2) edge[-] (bb);

\node[left=of b1] (a1) {$(a,a)$};
\draw (a1) edge[-] (abba);

\node[right=of b2] (a2) {$(a,a)$};
\draw (a2) edge[-] (bba);

\end{tikzpicture}
	\end{center}
\end{example}
Independently of their length, all words $(ab^na)_{n \geq 2}$ have 
this "summary" of height $4$ and size  $7$. 
The existence of "summaries" of constant depth is true in general, as shown in Theorem~\ref{thm:summary}.

This definition resembles the one of  Simon's factorization forests~\cite{simonforest}. However, an important difference is that Simon's trees are meant to factorize the entire word, while ours omit a lot of information by skipping intermediate idempotents. This lets us obtain better bounds on the height of the tree, since Simon's trees have linear height in the size of the semigroup, not just its "regular $\mathcal{J}$-length".
Those bounds are essential to obtain singly-exponential bounds, and then a polynomial-space algorithm, in the quantitative setting.

\begin{theorem}\label{thm:summary}
	For all $w \in \monoid^*$ there exists a "summary" whose result is  $\evalmorph{\monoid}(w)$ and of height
	at most $\Jlength{\monoid} (\log_2(|\monoid|) + 2\log_2(\Ramsey{\monoid}{3}) +4)$.
\end{theorem}

\begin{proof}[Proof sketch]
	The full proof is presented in Appendix~\ref{app:summary}.
	We first define the "regular $\mathcal{J}$-length@@element" of an element of the semigroup, as the one of the subsemigroup of elements $\Jleq$-below it.
	The proof goes through an induction on the "regular $\mathcal{J}$-length" of elements of the semigroup. 
	We start by cutting the word in minimal blocks of maximal "regular $\mathcal{J}$-length@@element". We factorize these blocks as a single letter and a block of smaller "regular $\mathcal{J}$-length@@element", for which we get a factorization by induction.
	Then we consider the word obtained by replacing each block with its value in the semigroup.
	
	We cut this word into infixes, each long enough to guarantee that it contains idempotents. We describe an operation that lets us merge some blocks where the same idempotent appears.
	We then use properties of the Green relations to bound the number of blocks obtained this way.
\end{proof}

\subsection{Application to the flow semigroup and iterations}
\label{subsec:applications}

This section is dedicated to the proof of two bounds which are crucial to show that the \sfp\ can be solved in polynomial space. The first one is Theorem~\ref{cor:sharpexpression}, which gives a polynomial upper-bound on the maximal depth of a $\sharp$-expression generating elements of $\F$. The second one is Theorem~\ref{thm:boundonflow},
which establishes a polynomial upper bound on the "optimal sequential flow", in case it is finite.

Let $n = |V|$, and let $\mathcal{B} = \mathbb{B}^{n \times n}$ be the finite semigroup of Boolean matrices of dimension $n$, equipped with the $(\lor, \land)$ matrix product.
The following result bounds its "regular $\mathcal{J}$-length" by a polynomial in its dimension.

\begin{theorem}[{\cite[Theorem 2]{Jecker21}}]
	The "regular $\mathcal{J}$-length" of $\mathcal{B}$ is at most $ (n^2+ n+ 2)/2$.
\end{theorem}

Let us now take a look at the "flow semigroup", 
$\F\subseteq \mmsm^{n \times n}$ with $\mmsm = (\set{0,1,\omega},\max,\min)$.

\begin{lemma}
	$\F$ is isomorphic to a sub-semigroup of $\?B^2$.
\end{lemma}

\begin{proof}
	Define $\psi : \F \to \mathcal{B}^2$ such that for all matrix $x \in \F$, $\psi(x) = (\mu_1,\mu_\om)$ with, for all $i,j \in [1,n]$,
	\[\mu_1(i,j) = \begin{cases}
		\top \text{ if } x(i,j) \geq 1\\
		\bot \text{ otherwise}
	\end{cases}
	\qquad\text{and}\qquad
	\mu_\om(i,j) = \begin{cases}
		\top \text{ if } x(i,j) = \omega\\
		\bot \text{ otherwise.}
	\end{cases}
	\]
	
	This is an injective function, and it is easily verified that this is a morphism.
\end{proof}

\begin{theorem}\label{thm:sef}
	Every word $w \in \F^*$ admits a "summary" of height
	at most $ 536 \cdot n^{10}$
	.
\end{theorem}

\begin{proof}
	As observed above, $\F$ is isomorphic to a subsemigroup of $\mathcal{B}^2$, which has "regular $\mathcal{J}$-length" bounded by $(n^2 +n +2)/2$.
	The "regular $\mathcal{J}$-length" of $\F$ is then at most $(n^2 +n +2)^2/4$, while its size is $3^{n^2}$.
	By Theorem~\ref{thm:Ramseybounds}, we obtain $\Ramsey{\F}{3} \leq 3^{(n^2+1) (n^2 +n +2)^2} \leq 3^{32n^6}$ for $n \geq 1$.
	Then, by Theorem~\ref{thm:summary} we have that every word over $\F$ has a "summary" of height at most 
	\begin{align*}
		\Jlength{\F} &(\log_2(|\F|) + 2\log_2(\Ramsey{\F}{3}) +4)\\ 
		\leq\quad &\frac{(n^2 + n + 2)^2}{4} (n^2\log_2(3) + 2(32n^6 \log_2(3)) +4) \\
		\leq\quad &4n^4 (2n^2 + 128n^6 +4) \\
		\leq\quad &536 n^{10} \qedhere
	\end{align*}
\end{proof}

We now need to integrate the $\sharp$ operator in this construction. We do this by following a proof of Simon~\cite[Theorem 9]{simontropical} on a different semigroup.

\begin{lemma}\label{lem:bound-chains-sharp}
	Let $m \geq n^2$ and $e_1, \dots, e_m \in\idempotents{\F}$ idempotents of $\F$ such that $e_{i+1} \Jleq e_i^\sharp$ for all $i$.
	Then there exists $i$ such that $e_i^\sharp = e_i$.
\end{lemma}
The proof of this lemma is presented in Appendix~\ref{app:sharp-height}.
Observe that Example~\ref{ex:nested} can be generalized to obtain instances where we need a linear number of nested $\sharp$ in order to obtain some elements of $\F$.

	That result has two interesting consequences,
	which are the keys to obtain \PSPACE\ algorithms for the \sfp.
	The first consequence is that small ``$\sharp$-expressions'' (Definition~\ref{def:sharpexpression}) are enough to generate all elements in the flow semigroup, as stated in Theorem~\ref{cor:sharpexpression}.
	The formal proof of Theorem~\ref{cor:sharpexpression} is presented in Appendix~\ref{app:bound-sharp-exp-height}.
	The idea is as follows: take a "$\sharp$-expression" generating an element of $\F$. First use Lemma~\ref{lem:bound-chains-sharp} to eliminate redundant idempotent nodes and reduce its $\sharp$-height below $n^2$.
	Then reorganise its product nodes to obtain balanced subtrees of product nodes. Since the monoid has exponential size in $n$, trees of polynomial height suffice to obtain everything we can with products.

The second consequence of Lemma~\ref{lem:bound-chains-sharp} is that every capacity word can be 
represented as a small "$\sharp$-summary":
We define \emph{$\sharp$-summaries}, where idempotent nodes for $e \in \idempotents{\monoid}$ are labeled with $e^\sharp$ instead of $e$. 

\begin{definition}\label{def:sharp-sum}
	A ""$\sharp$-summary"" of a word $w\in \F^*$ is a $\F \times \F^*$-labeled ordered binary tree, with three types of nodes:
	\begin{itemize}
		\item  A leaf is labeled by $(x, x)$ for some $x \in \F$,
		\item A product node has two children. If their labels are $(x_1,u_1)$ and $(x_2,u_2)$ then its label is $(x_1 \cdot x_2, u_1 u_2)$
		\item An idempotent node has two children. If their labels are $(x_1,u_1)$ and $(x_2,u_2)$ then $x_1=x_2 =e$ is an "idempotent" of $\F$ and the label of the node is $(e^\sharp, u_1 w_1 \dots w_m u_2)$ for some $w_1, \dots, w_m \in \F^*$ such that all $w_i$ have a "$\sharp$-summary" whose root is labeled $(e,w_i)$.
		In the case that $e = e^\sharp$ we say that the node is a ""stable idempotent node"", and an ""unstable idempotent node"" otherwise.
	\end{itemize}
Moreover, the root is labeled by $(x, w)$ for some $x \in \F$, which is called the ""result@@sum"" of the "$\sharp$-summary".
\end{definition}

\begin{example}\label{ex:tropicalshortendfactorisationexample}
In Example~\ref{ex:intro}, since $b=b^2$, 
the following tree is a "$\sharp$-summary" of $ab^na$.
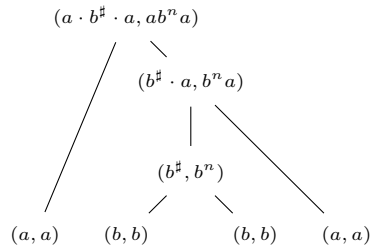
	\begin{center}
	\begin{tikzpicture}[node distance=0.25cm and 0.6cm]

\node (abba) {$(a\cdot b^\sharp \cdot a,ab^na)$};

\node[below right=of abba] (bba) {$(b^\sharp \cdot a,b^na)$};
\draw (bba) edge[-] (abba);

\node[below=of bba] (bb) {$(b^\sharp ,b^n)$};
\draw (bb) edge[-] (bba);

\node[below left=of bb] (b1) {$(b,b)$};
\node[below right=of bb] (b2) {$(b,b)$};
\draw (b1) edge[-] (bb);
\draw (b2) edge[-] (bb);

\node[left=of b1] (a1) {$(a,a)$};
\draw (a1) edge[-] (abba);

\node[right=of b2] (a2) {$(a,a)$};
\draw (a2) edge[-] (bba);

\end{tikzpicture}
	\end{center}
\end{example}
This $\sharp$-summary bears some similarity with the summary
of Example~\ref{ex:shortendfactorisationexample}.
However, there is a crucial difference: the root of the $\sharp$-summary
is labelled by
$a\cdot b^\sharp\cdot a$, which is an unboundedness witness
(see details before Definition~\ref{def:unboundednesswitness}).
This is not the case in the summary of Example~\ref{ex:shortendfactorisationexample}: the root is labelled 
by $z=aba$ which is not an unboundedness witness
since $z(v_1,v_4)=1$.
 
Independently of their length, all words $(ab^na)_{n \geq 2}$ have 
this $\sharp$-summary of height $4$ and size $7$. The existence of "$\sharp$-summaries" of constant height (and size), whatever the size of the word, is true in general, as shown in Theorem~\ref{thm:tsef}.

Our next step is to show that all words have a "$\sharp$-summary" of polynomial height in the number of vertices $n = |V|$.
We take inspiration from two existing proofs: First, one by Kirsten~\cite{DBLP:journals/ita/Kirsten05} to show that the number of unstable nodes along a branch of a "$\sharp$-summary" (or a "$\sharp$-expression") is bounded by a polynomial in $n$, the number of vertices\footnote{A polynomial bound in $\mathcal{O}(n^4)$ can be obtained by proving that $e^\sharp \Jlt e$ for all $e \in \F$, and using the bound on the "regular $\mathcal{J}$-length". We prefer to use Lemma~\ref{lem:bound-chains-sharp}, which gives $\mathcal{O}(n^2)$.}.
Second, one by Simon~\cite{simonforest} to show that every word has a "$\sharp$-summary" where the distance between consecutive unstable nodes along a branch is bounded by another polynomial in $n$.
Adapting and combining those two arguments yields the result.
The proof is in Appendix~\ref{app:tsef}.

\begin{theorem}\label{thm:tsef}
For all $w \in \F^*$ there is a "$\sharp$-summary" of height at most $\tropicalshortendfactorizationupperbound$.
\end{theorem}

We make use of Theorem~\ref{thm:tsef} to prove Theorem~\ref{thm:boundonflow}.
By the max flow-min cut theorem~\cite{ford1956maximal}, to prove a bound on the maximal flow of capacity words it suffices to find for each one a cut of cost at most this bound.
We construct this cut by induction on the height of a "$\sharp$-summary" for the capacity word.
A key part of the argument is the trichotomy from Lemma~\ref{lem:idempotent}, particularly case (ii). 
When dealing with an idempotent node we show that for all $v,v'$, either iterating the corresponding idempotent gives us unbounded flows from $v$ to $v'$, or, in any iteration of the idempotent, we can find a cut of bounded cost between $v$ and $v'$ within the first and last iterations.
This justifies the abstraction of "$\sharp$-summaries": in a sequence of iterations of an idempotent we keep only the first and last.
The value of the constructed cut is exponential in the height of the "$\sharp$-summary", which yields the result by Theorem~\ref{thm:tsef}.
The full proof is in Appendix~\ref{app:bound-on-flow}.

\section{Extensions}
Our approach to solve the \sfp\ can be extended to further generalisations that consider sequential flows between sets of source and target vertices and under regular constraints on the witnessing capacity words.

\subsection{Fair flows along multiple edges / out of multiple sources}\label{sec:fair}

The previous results and algorithms can be adapted to solve a more 
general problem.

Instead of a single source-target pair $(v_s,v_t)$,
the problem comes with a collection $E \subseteq V\times V$ of edges, and one wishes to carry as much flow as possible along those edges.

That is, the objective is \emph{not} to maximize the total amount of flow through the edges
(this can easily be reformulated as an instance of the \sfp,
solutions to which then may lead to one of the edges being unused).
Instead, we ask to maximize the minimal "global flow" among all given edges.

Using the notation from Definition~\ref{def:tokenflow} of 
"token flow" $d$ and its associated "global flow" $g(d))$,
the
\intro*\msfp{}
asks to compute the value $\sup_{\text{token flow $d$}}\abs{d}$, where
\[
\abs{d} = \min\set{g(d)(v_s,v_t) \mid (v_s, v_t) \in E}.
\]

This problem can be solved similarly to the \sfp,
except one looks for a different kind of witnesses in the flow semigroup $\F$.
\begin{definition}
	A ""fair unboundedness witness"" is an element of $x \in \F$
	such that
	\[
	\forall (v_s,v_t) \in E, x(v_s,v_t) = \omega.
	\]
\end{definition}

\begin{theorem}\label{thm:msfp}
	The \msfp\ can be solved in polynomial space.
\end{theorem}
\begin{proof}
The unboundedness of the \msfp\ is equivalent to the existence
of a "fair unboundedness witness" in the flow semigroup $\F$,
the proof is a straightforward adaptation of the proofs of
Lemma~\ref{lem:sufficientwitness} and~\ref{lem:necessarywitness}.
Such a witness can be looked for in polynomial space using Algorithm~\ref{algoquali}.
If no such witness exists then 
the value $\sup_d |d|$ is finite 
and even bounded by $B=\boundV$ according to Theorem~\ref{thm:boundonflow}.
Then a variant of Algorithm~\ref{algo}
can be used in order to optimise 
$\abs{d}=\min\set{g(d)(v_s,v_t) \mid (v_s, v_t) \in E}$
by looking for a token flow $d$
moving exactly $k$ tokens along every edge in $E$,
where $k$ is optimized by dichotomic search in the interval
$[0, B]$.
\end{proof}

\begin{remark}
If, instead of simultaneous flow across a subset of edges, one is interested in checking simultaneous flows out of several sources into the target, the resulting problem easily reduces to the \msfp.
Indeed, it suffices to introduce a new target vertex $t$ and a new final capacity constraint that transfers tokens from all target states to $t$.
Then the multi-source flow problem is an instance of the \msfp\ where we ask to maximise the simultaneous flow from each source to $t$.
\end{remark}

\subsection{Regular constraints}

The \intro*\rmsfp\ generalises the \msfp\ by requiring that we only consider "capacity words" within a given regular language.
Formally, 
consider a finite set $\Capas \subseteq \left(\NN\cup\{\omega\}\right)^{\vertices\times \vertices}$ of capacities,
a set $E \subseteq V\times V$
and a regular language $L \subseteq \Capas^*$ recognized by a non-deterministic automaton with $m$ states. 
Let $n=\abs{\vertices}$ and $K$ be the largest finite capacity in any element of $\Capas$.
The problem is to compute the optimal fair "token flow" over capacity words in $L$, i.e.,
\[
\sup_{w\in L}\quad\sup_{\substack{\text{"token flow" $d$}\\ \text{over $w$}}}\quad|d|\enspace.
\]

\begin{theorem}\label{theo:regularSFP}
The \rmsfp\ can be solved in polynomial space.
Furthermore, if the answer is bounded then it is at most  $K(2|V|)^{(170\log_2(m)+835)|V|^{12}}$.
\end{theorem}

\newcommand{\BA}{\mathcal{B}_{\mathcal{A}}}

This can be shown with the technique discussed in Sections~\ref{sec:seqflow} and~\ref{sec:fair} with a slightly extended definition of the "flow semigroup".
We describe the necessary adjustments below; more detailed proofs are presented in
Appendix~\ref{app:reg}. %

\medskip
Fix a non-deterministic finite automaton $\mathcal{A}$ with $Q$ the set of $m$ control states, $I,F\subseteq Q$ sets of initial and final states, and $\Delta\subseteq Q\x\Capas\x Q$ the transitions over the alphabet $\Capas$.
We define the semigroup $\BA$ made of the set of triples in $Q \times \{0,1,\omega\}^{V \times V} \times Q$, plus an element $\bot$, and
where the product $\star$ is defined as follows: 
\[
(q_1, x, q'_1) \star (q_2, y, q'_2) = 
\begin{cases}
	&(q_1, x\cdot y, q'_2) \text{ if } q'_1 = q_2\\
	&\bot \text{ otherwise.}
\end{cases}
\]
We set $\bot \star \bot = \bot$ and $(q,x,q') \star \bot = \bot \star (q,x,q') = \bot$ for all $(q,x,q') \in Q \times \{0,1,\omega\}^{V \times V} \times Q$.
 
To lift the iteration operation,
notice that, except for $\bot$, an idempotent of $\BA$ is of the form $(q,e,q)$ with an idempotent matrix $e$.
Define $(q, e, q)^\sharp = (q, e^\sharp, q)$ for all $q \in Q$ and idempotent $e$, and let $\bot^\sharp = \bot$.

\knowledgenewrobustcmd{\FA}{\cmdkl{\mathcal{F}_{\mathcal{A}}}}

\AP The ""labeled flow semigroup"" $\intro*\FA$ is defined as the smallest sub-semigroup of $\BA$ which contains $\set{(q,x_a, q') \mid (q,a,q') \in \Delta}$ and is stable under $\star$ and $\sharp$.

It then suffices to go through the same steps as in Sections~\ref{sec:seqflow} and~\ref{sec:diving}, with some minor changes to accommodate the state constraints. We explain the necessary changes in Appendix~\ref{app:reg}.

\section{Conclusion}

We provide a new algebraic technique to solve the \sfp\ in \PSPACE.
We mention two promising directions to utilize the results shown here.
First, we aim to adapt our techniques to graphs generated by graph grammars.
Second, we plan to extend our framework to settings with asynchronous flows, with applications to asynchronous distributed computing.

\bibliography{journals,conferences,biblio.cleaned}

\newpage
\appendix

\section*{Appendix}

\section{Proofs of Section~\ref{sec:seqflow}}

\subsection{Proof of Lemma~\ref{lem:idempotent}}

\begin{proof}[Proof of Lemma~\ref{lem:idempotent}]
The three cases are clearly distinct:
in case i) $K_n$ is $\om$,
in case it is finite but unbounded,
in case iii) it is bounded.

Case i).
Assume that $e(v,v')=\omega$ and let $n\geq 1$.
Since $e^n=e$, then $e^n(v,v')=\omega$.
By definition of the product,
there exists a sequence $v=v_0,v_1,\ldots, v_n=v'$
such that $e(v_{i-1},v_{i})=\omega, i \in 1\ldots n$.
Those edges can carry an arbitrarily large amount of flow, 
hence the first statement of i) holds. 
Moreover,
for $n=|V|$, there must exist
$0\leq i < j \leq n$ such that $v_i=v_j$. Set $v''=v_i=v_j$.
Then $e(v'',v'')=e^{j-i}(v'',v'')\geq 
\min(e(v_i,v_{i+1}), \ldots, e(v_{j-1},v_j)) = \omega$.
Conversely, if such a pair $v',v''$ exists then 
$e(v,v')=e^2(v,v')\geq
\min(e(v,v''),e(v'',v')))=\om$.

Cases ii) and iii).
Assume now that $e(v,v')=1$ and let $n\geq 1$.
Analogously to the previous case
we must have $\forall n, K_n\geq 1$.

Introduce 
\[P=\{(v_0,v'_0) \in V^2 \mid e(v,v_0)=\omega \land e(v_0,v'_0)= 1 \land e(v'_0,v')= \omega\}.\]
The alternative conditions for cases (ii) and (iii) simply check whether $P$ is non-empty or empty, respectively.
We start with a preliminary remark.
Let $d_n$ be a flow from $v$ to $v'$ in the pipeline $e^n$.
Since $e(v,v')=1$,
every sequence $v=v_0,v_1,\ldots, v_n=v'$ such that all $d_n(v_i,v_{i+1})>0$ carry positive flow
must traverse a $1$-labelled edge $e(v_0,v'_0)=1$ at some point.

This remark shows $K_n \leq n|V|$: 
the cut which removes all edges of capacity $1$
disconnects $v$ from $v'$. This has cost at most $n|V|$
thus according to the maxflow-mincut theorem,
$K_n \leq n|V|$.

Assume first that $P$ is empty.
Consider the cut which removes from the pipeline all $1$-labelled edge
in the first and last copy of $e$. Since $P$ is empty, 
according to the previous remark,
the cut disconnects
$v$ from $v'$.
There are at most $2|V|$ such edges thus according to the maxflow-mincut theorem,
$K_n \leq 2|V|$.

Assume now that $P$ is not empty.
A reasoning similar to case i) shows that $P$ contains a pair
 $(v_0,v'_0)$ such that $e(v_0,v_0)=e(v'_0,v'_0)=\omega$.
Consider the following "sequential flow" of value $n$ in the pipeline $ee^ne$.
First all the flow traverses the $\om$-edge $(v,v_0)$.
Then at every step $1 \leq \ell \leq n$,
one unit of flow traverses the $1$-edge $(v_0,v'_0)$
while the rest of the flow stays on the $\omega$-loops on $v_0$ and $v'_0$.
Finally all the flow traverses the $\om$-edge $(v'_0,v')$.
This shows that $K_{n+2} \geq n$.
\end{proof}

\subsection{Proof of Lemma~\ref{lem:sharpsharp}}
\label{app:sharpsharp}

\begin{proof}[Proof of Lemma~\ref{lem:sharpsharp}]
The iteration operation does not create new unstable edges.
Let $(v,v')$ such that $e^\sharp(v,v') = 1$.
To prove stability, we take any
$v_0,v_0' \in V$ such that 
$\omega=e^\sharp(v,v_0)=e^\sharp(v'_0,v')$
and prove that $e^\sharp(v_0,v'_0)=0$,
which is enough according to iii) of Lemma~\ref{lem:idempotent}.
By definition of $e^\sharp$, it is also enough to prove $e(v_0,v'_0)=0$.

The first case is $\omega=e(v,v_0)=e(v'_0,v')$.
Since $e^\sharp(v,v') = 1$ then $e(v,v') = 1$ hence $(v,v')$ is 
"stable@@pair" in $e$. According to iii) of Lemma~\ref{lem:idempotent},
$e(v_0,v'_0)=0$.

The second case is $1=e(v,v_0)$.
Then $(v,v_0)$ is "unstable@@pair" in $e$ thus there exists $w_0,w'_0$
such that $\omega=e(v,w_0)=e(w'_0,v_0)$ and $e(w_0,w'_0)=1$.
Since $(v,v')$ is "stable@@pair" in $e$ then according to iii) of Lemma~\ref{lem:idempotent}, $e(w_0,v_0')=0$.
Then, 
\begin{align*}
	0 &= e(w_0,v_0') = e^3(w_0,v_0') \\
	  &\geq \min(e(w_0,w'_0),e(w'_0,v_0),e(v_0,v'_0)) \\
	  &= \min(1,e(v_0,v'_0))\\
	  &= e(v_0,v'_0).
\end{align*}
The third case, $1=e(v'_0,v'$), is analogous.
\end{proof}

\subsection{Proof of Lemma~\ref{lem:sufficientwitness}}\label{app:sufficient}

We prove the following property on elements of the "flow semigroup" $\F$:

\begin{itemize}
\item[\intro*\emph{\diamondproperty}\hspace{-3mm}]
For all $x \in \F$, for all $N \in \NN$, there exist a "capacity word" $\capa$ and a "token flow" $\dsf$ over $\capa$ such that for all $v, v' \in V$, the following two conditions are satisfied:
\begin{enumerate}
	\item $x(v,v') = \omega$ $\Rightarrow$ $\gsf(\dsf)(v,v') \geq N$
	
	\item $x(v,v') \geq 1$ $\Rightarrow$   there is a "path@@pipeline" in $\capa$ from $v$ to $v'$ (in the sense of Definition~\ref{def:pathpipeline}).
\end{enumerate}
\end{itemize}

This clearly implies the lemma, by applying property \diamondproperty (specifically, the first item) to an "unboundedness witness".

The proof follows the inductive definition of the "flow semigroup" $\F$: first prove property \diamondproperty for the abstracted capacity constraints $\set{x_a \mid a \in A}$,
then show that property \diamondproperty is preserved by product and iteration.

Let $x_a$ be the abstraction  of a capacity constraint $a \in \NN^{V \times V}$. Let $N \in \NN$. 
We simply use the "capacity word" of length one $\capa = a$, with a "token flow" $\dsf$ over $\capa$ transferring $N$ tokens from $v$ to $v'$ for all $v,v' \in V$ such that $x_a(v,v') = \omega$.
Both conditions 1) and 2) clearly hold by definition of the abstraction $x_a$, hence the property \diamondproperty holds.

\medskip

We now show that property \diamondproperty is preserved by the product operation. 
Let $x,y \in \F$ which both satisfy property \diamondproperty, and consider their product $x\cdot y$ and some integer $N$.
To obtain the property for $N$ for $x\cdot y$, we apply it with $N'=|V|^2N$ for $x$ and $y$, and obtain two capacity words $\capa_x, \capa_y$
and token flows $\dsf_x, \dsf_y$ which satisfy conditions 1) and 2) 
for $N'$.
We use the concatenation $\capa = \capa_x\capa_y$ as our "capacity word".
By definition of the product,
for each $v,v'$ such that $(x\cdot y)(v,v') \geq1$, there exists $\overline{v} \in V$ such that $x(v,\overline{v}) \geq 1$ and $y(\overline{v},v')\geq 1$.
Since condition 2) holds for both $\capa_x$ and $\capa_y$, there is a "path@@pipeline" from $v$ to $\overline{v}$ in $\capa_x$, and from $\overline{v}$ to $v'$ in $\capa_y$, and thus from $v$ to $v'$ in $\capa$.
We now define the "token flow" for $x\cdot y$.
Take an $\omega$-edge of $(x \cdot y)$ i.e. a pair $v,v'$ such that $(x\cdot y)(v,v') =\omega$.
By definition of the product, there exists $\overline{v} \in V$ such that $x(v,\overline{v}) = y(\overline{v},v')= \omega$. 
We select $N$ tokens from the token flow $\dsf_x$ going from $v$ to $\overline{v}$ and $N$ tokens from $\dsf_y$ going from $\overline{v}$ to $v'$. Since at least $|V|^2N$ tokens are moving along each $\omega$-edge in $\dsf_x$ and $\dsf_y$, the selection can be performed so that the set of tokens are pairwise-disjoint.
The $N$ tokens selected for $(\overline{v},v')$ in $\dsf_y$ are renamed so that they are equal to the $N$ tokens selected for $(v,\overline{v})$ in $\dsf_x$. All other tokens are deleted from $\dsf_x$ and $\dsf_y$. The resulting "token flows" are concatenated, to get one over $\capa$ where $N$ tokens are moving along every $\omega$-edge of $x \cdot y$. This terminates the inductive step for the product.

\medskip

Finally, we prove that the property \diamondproperty is preserved by iteration, i.e. application of the $\sharp$ operator to idempotent elements.
To begin with, we show that it suffices to prove that the property \diamondproperty is preserved while applying $\sharp$ on "simple unstable idempotents", defined below.

\begin{definition}
	A ""simple unstable idempotent"" is one which only has self-loops ($1$ or $\omega$), plus a single "unstable@@pair" pair.
	Formally, it is an idempotent element $e \in \F$ such that there exist $v \neq v' \in V$ with $e(v,v') = 1$ and for all $\overline{v} \neq \overline{v}' \in V$, $(v,v') \neq (\overline{v}, \overline{v}') \Rightarrow e(\overline{v},\overline{v}') = 0$.
\end{definition}

Given two elements $x,y$ of $\F$, we write $x \leq y$ if $x(v,v') \leq y(v,v')$ for all $v,v' \in V$.
Clearly, if $y$ satisfies the property \diamondproperty, then so does $x$ (take the same pipelines and token flows).

\begin{lemma}\label{lem:decompose-idempotent}
	For every "unstable@@idempotent" idempotent $e\in \F$, there exist some "simple unstable idempotents" $e_1, \ldots, e_m \leq e\in \F$ such that  $e^\sharp \leq e e_1^\sharp  \cdots e_m^\sharp e$.
\end{lemma}

\begin{proof}
	For every "unstable pair" $(\vb,\vb')$ of $e$ such that $e(\vb,\vb) = e(\vb',\vb') = \omega$, let $e_{\vb,\vb'}$ be the matrix such that $e_{\vb,\vb'}(\vb,\vb')=1$, $e_{\vb,\vb'}(v,v) = e(v,v)$ for all $v \in V$ and all other pairs are mapped to $0$.
	It is easily verified that $e_{\vb,\vb'}$ is a "simple unstable idempotent" and that $e_{\vb,\vb'}^\sharp$ is the idempotent identical to $e_{\vb,\vb'}$ except that $e_{\vb,\vb'}^\sharp(\vb,\vb')=\omega$.
	
	Let $m$ the number of "unstable@@pair" pairs of $e$ such that $e(\vb,\vb) = e(\vb',\vb') = \omega$ and $e_1, \ldots, e_m$ be the associated "simple unstable idempotents", in any order.
	Define $x= e e_1^\sharp \dots e_m^\sharp e$, we show that $e^\sharp \leq x$.
	Since $e$ is idempotent, then $e^3=e$
	and by definition of the product,
	for all $v,v' \in V$, there exist $v_0$ such that $e(v,v')=\min(e(v,v_0), e(v_0,v_0) , e(v_0,v'))$. For every $\ell\in 1\ldots m$, $e(v_0,v_0)=e_\ell(v_0,v_0) \leq e_\ell^\sharp(v_0,v_0)$.
	Therefore, by definition of the product, 
	\[
	e(v,v') \leq \min(e(v,v_0), e_1^\sharp(v_0,v_0) , \ldots, e_m^\sharp(v_0,v_0),e(v_0,v')) \leq x(v,v').
	\]
	This holds for every edge $v,v'$, thus $e \ldots x$.
	
	To show that $e^\sharp \leq x$, it remains to prove that all "unstable@@pair" pairs of $e$ are mapped to $\omega$ in $x$. 
	Let $(v, v')$ be one. By definition of (un)stability (condition ii) in Lemma~\ref{lem:idempotent}), there exists $v_0, v_0'$ such that $e(v,v_0) = e(v_0,v_0) =  e(v_0', v_0') = e(v_0',v') = \omega$ and $e(v_0,v_0')=1$.
	
	As a consequence, there is an index $\ell$ such that $e_\ell = e_{\vb,\vb'}$ and thus $e_\ell^\sharp(\vb,\vb') = \omega$. 
	By definition of the product, $x(v,v')$ is greater or equal than
	\[ \min(e(v,v_0), e^\sharp_1(v_0,v_0) ,\ldots, e^\sharp_\ell(v_0,v'_0), \ldots, e^\sharp_m(v'_0,v'_0) ,e(v,'_0, v')
	\] 
	and all those terms are $\omega$ hence $x(v,v')\geq \omega$.
	Finally, $e^\sharp \leq x$.
\end{proof}

We now focus on "simple unstable idempotents".

\begin{lemma}\label{lem:simple-unstable-preserve}
Let $e$ a "simple unstable idempotent" with $(v,v')$ its "unstable@@pair" pair.
Suppose that for all $N$ there is a "capacity word" $\capa$ such that 
\begin{itemize}
	\item there is a "token flow" transferring $N$ tokens along every $\omega$-loop for all $N$, and
	
	\item there is a "path@@pipeline" in $\capa$ from $v$ to $v'$.
\end{itemize}

Then, for all $N$, there is a "token flow" over a "capacity word" $w'$, transferring $N$ tokens along every $\omega$-loop, plus $1$ from $s$ to $t$.
\end{lemma}

\begin{proof}
	Let $N \in \NN$. 
By hypothesis, there exists a capacity word $\capa$,
a "token flow" $d$ on $\capa$ transferring $N+2$ tokens along every $\omega$-loop, and a "path@@pipeline" $\pi$ from $v$ to $v'$ in $w$.
Since $(v,v')$ is "unstable@@pair" and $e$ is a "simple unstable idempotent", there are $\omega$-loops on both $v$ and $v'$.

We say that a token $\tau$ crosses path $\pi$ at date $\delta$ if the $\delta+1$th state in $\pi$ is the $\delta+1$th state visited by $\tau$.

We cut $d$ as follows: 
Let $0 = \delta'_0 < \delta_1 \leq \delta'_1 < \delta_2 < \dots < \delta'_m$ be such that $\delta_i$ is the first date at which a token $\tau_i$ crosses path $\pi$ strictly after $\delta'_{i-1}$, 
$v_i$ is the initial vertex of $\tau_i$ and $\delta'_i$ is the last date at which a token $\tau'_i$ with initial vertex $v_i$ crosses $\pi$.

Define $d'$ the "token flow" where tokens $\tau_i$ and $\tau'_i$ have been deleted.
Let $\tau$ be a fresh token.
For each $i \in [0,m]$ let $d_i$ be the "token flow" such that before $\delta'_i$ we make $\tau$  mimic $\tau'_i$ (if $i>0$), between $\delta'_i$ and $\delta_{i+1}$ we make $\tau$ follow $\pi$ and after $\delta_i$ we make $\tau$ mimic $\tau_{i+1}$ (if $i<m$).

Observe that $d_i$ transfers at least $N$ tokens along each $\omega$-loop, plus $\tau$ which is transferred from $v_i$ to $v_{i+1}$.

The composition of these flows $d_0 \cdots d_m$ transfers $N$ tokens along each $\omega$-loop, plus $\tau$ which is transferred from $v_0 = s$ to $v_{m+1} = t$. 
\end{proof}

\begin{lemma}
	Let $e$ a "simple unstable idempotent" with $(v,v')$ its "unstable@@pair" pair.
	Let $\capa$ a capacity word  such that 
	\begin{itemize}
		\item there is a "token flow" transferring $N$ tokens along every $\omega$-loop for all $N$, and
		
		\item there is a path in $\capa$ from $v$ to $v'$
	\end{itemize}
	
	Then there is a "token flow" transferring $N$ tokens along every $\omega$-loop, plus $N$ from $v$ to $v'$, for all $N$.
\end{lemma}

\begin{proof}
	Take a "token flow" transferring $N$ tokens along every $\omega$-loop, plus one from $v$ to $v'$ (it exists by the previous lemma).
	By iterating it $N$ times, we transfer $N$ tokens from $v$ to $v'$, and still $N$ tokens along every $\omega$-edge.
\end{proof}

To conclude, let $e$ be an idempotent satisfying the property \diamondproperty, and let $e_1, \ldots, e_m$ be as in Lemma~\ref{lem:decompose-idempotent}. Since $e_i \leq e$ for all $i$, each $e_i$ satisfies the property \diamondproperty.
By Lemma~\ref{lem:simple-unstable-preserve}, so does each $e_i^\sharp$, hence also $ee_1^\sharp \dots e_m^\sharp e$ since property \diamondproperty is preserved by product.
As $e^\sharp \leq ee_1^\sharp \dots e_m^\sharp e$, we obtain that $e^\sharp$ satisfies property \diamondproperty.

We have shown that property \diamondproperty holds on $\set{x_a \mid a \in A}$, and is preserved by product and application of $\sharp$. As a result, it holds on all elements of $\F$.

\subsection{Proof of Lemma~\ref{lem:algoquali}}
\label{app:algo-sharp-exp}

By induction on $h\ge 0$. %
\emph{Case $h=0$}. The only valid expressions of height $1$ are (abstractions of) single capacity letters. In the algorithm then correctly returns whether this is the case for $x$,
either on line \ref{alg-base-1} or \ref{alg-base-2}.
\emph{Case $h\geq 1$}.
If there is a $\sharp$-expression of height $h\geq1$ then its root is either labelled 
by a product $x_1\cdot x_2$ and there are $\sharp$-expressions of height $h-1$ for the two $x_1$, or the root is labelled by $e^\sharp$ for an idempotent $e$ and there exists a $\sharp$-expression of height $h-1$ for $e$.
Either way, the algorithm correctly returns \emph{true} in line \ref{alg-true-1} or \ref{alg-true-2}.
Otherwise, if no such expression exists, the algorithm correctly 
returns \emph{false} in line \ref{alg-end}.

\subsection{Proof of Theorem~\ref{theo:quantitativealgorithm}}
\label{app:pspace-algo}

	By Theorem~\ref{theo:qualitativealgorithm}, we can check in polynomial space whether the "optimal sequential flow" is unbounded.
	Suppose now that it is bounded and let $n=|V|$.
	
	We proceed by binary search (see Algorithm~\ref{algo}).
	According to Theorem~\ref{thm:boundonflow} and the constant fixed in line 1,
	Algorithm~\ref{algo} preserves an invariant:
	the finite "optimal sequential flow" belongs to the interval $[m, M-1]$.
	Since this interval is halved at every step of the loop,
	this loop terminates in at most $\left\lceil \log_2\left(\bound\right)\right\rceil$ steps.
	
	The loop requires to test the existence of a "sequential flow" of value $C$ in polynomial space.
	We make use of two notions.
	A $C$-configuration is a map $V \to \set{1,\dots, C}$ counting the number of tokens per vertex,
	whose sum of coordinates is $C$.
	For two $C$-configurations $x_s,x_t$,  
	a sequential integral flow from $x_s$ to $x_t$ of length $\ell$
	is a sequence $f_1f_2\ldots f_\ell\in \left(0\ldots C\right)^{V\times V}$ satisfying conditions~\eqref{scapcon} and~\eqref{sflowcon} and such that $\forall v\in V, x_s(v) = \outt(f_1)(v)$
	and $x_t(v) = \inn(f_\ell)(v)$.
	If such a sequential integral flow exists, we write $x_s \to_\ell x_t$.
	
	Since our finite capacities have integer values the integral flow theorem~\cite{FF1958} guarantees
	that the existence of "sequential flow" of value $C$ is equivalent 
	to the existence of $\ell$ such that $x_s \to_\ell x_t$,
	where $x_s$ has coordinate $C$ on $v_s$ and $0$ elsewhere
	and $x_t$ has coordinate $C$ on $v_t$ and $0$ elsewhere.
	
	Notice that
	$x_s \to_\ell x_t$ iff there exists a sequence $x_s \to_1 x_1 \to_1 \ldots \to_1  x_{\ell} \to_1 x_t$.
	Therefore, checking $x_s \to_\ell x_t$ amounts to a reachability question 
	in the space of $K$-configurations, for the relation $\to_1$.
	There is no need to traverse twice the same $C$-configuration,
	hence the search can be limited to $\ell \in 1\ldots C^{n}$.
	Checking whether $x \to_1 x'$ reduces to solving $|A|$ instances
	of \mfp,
	in polynomial time,
	details are given in Appendix~\ref{app:algotokenflow}.
	
	Thus, checking $x_s \to_\ell x_t$ with $\ell \in 1\ldots C^{n}$
	can be performed non-deterministically in \PSPACE, 
	by storing triplets of $C$-configurations and the length $\ell$.
	
	A deterministic \PSPACE\ algorithm is given by Savitch's theorem~\cite{savitch}: it consists in implementing a dichotomic search for the path, which requires only polynomial space in
	$n=|V|$, $|A|$ and $\log_2(C) \leq \log_2\left(\bound\right)$.

\newcommand{\OLY}{\vectV}

\begin{algorithm}
	\caption{\label{algo}
		Computing the optimal sequential flow}
	\begin{algorithmic}[1]
		
		\State{$M \gets \boundV + 1$}
		\State{$m \gets 0$}
		
		\Repeat
		\State{$C \gets \lceil(m+M)/2\rceil$}
		\If{there is a sequential flow $f$ such that $|f|=C$}
		\State{$m \gets C$}
		\Else
		\State{$M \gets C$}
		\EndIf
		\Until{$M \leq m +1$}
		\State \Return {$m$}
	\end{algorithmic}
\end{algorithm}

\subsection{Checking whether $x_s \to_\ell x_t$ in polynomial space}
\label{app:algotokenflow}

Let $x_s,x_t$ two $C$-configurations.

In case $\ell=1$, $x_s \to_1 x_t$ reduces to solving $|A|$ instances of \mfp\ with maximal capacity $C$, which can be solved in time
polynomial in $\abs{V}$ and the length of the binary representation of $C$.
For every $a\in A$, the \mfp\ instance is denoted $I(a,x_s,x_t)$. It is obtained by adding a new source vertex $v_s'$
with capacity constraints $x_s(v)$ on each edge $(v_s',v), v \in V$ and 
a new target vertex $v_t'$ with capacity constraints $x_t(v)$ on each edge $(v, v_t'), v \in V$,
and checking the existence of a flow of maximal value $C$ from $v'_s$ to $v'_t$.

The existence of a sequential integral flow is performed by  Algorithm~\ref{algotokenflow}.
The calling depth is $\leq \left\lceil \log_2(\ell)\right\rceil$%
hence it can be implemented using a stack of up to $\left\lceil\log_2(\ell)\right\rceil$ triplets of $C$-configurations and binary counter $\leq \ell$.
This requires polynomial space in $\log_2(\ell)$,
$|V|$ and $\log_2(C)$, plus the space for the $\ell=1$ case,
which is polynomial in $|A|$, $|V|$ and $\log_2(C)$.

\begin{algorithm}[ht]
	\caption{\label{algotokenflow}
		Checking whether $x_s \to_\ell x_t$}
	\begin{algorithmic}[1]
		\Function{IsIntSeqFlow}{$x_s$,$x_t$, $\ell$}
		\If{$\ell = 0$}
		\State\Return $x_s=x_t$
		\EndIf
		\If{$\ell = 1$}
		\For {every capacity $a\in A$}
		\If{the instance $I(a,x_s,x_t)$ has a flow of value $C$}
		\State\Return true
		\EndIf
		\EndFor
		\State \Return false
		\EndIf
		\For {every $C$-configuration $x$}
		\State $b_s \gets$ IsIntSeqFlow($x_s$,$x$,$\lceil\ell / 2\rceil$)
		\State $b_t \gets$ IsIntSeqFlow($x$,$x_t$,$\lfloor\ell / 2\rfloor$)
		\If{ $b_s$ and $b_t$ }
		\State \Return true
		\EndIf
		\EndFor
		\State \Return false
		\EndFunction
	\end{algorithmic}
\end{algorithm}

\begin{remark}
	We obtain an alternate way to check unboundedness of flow values: simply check whether there is a "sequential flow" of "value" $\boundV+1$ (see the proof of Theorem~\ref{theo:quantitativealgorithm} for a way to do this in polynomial space). If no such "sequential flow" exists, then we have a bound, otherwise there exist "sequential flows" with unbounded "values" by Theorem~\ref{thm:boundonflow} and~\ref{thm:characterisation}.
\end{remark}

\section{Proofs of Section~\ref{sec:diving}}

\subsection{Proof of Theorem~\ref{thm:summary}}\label{app:summary}

\knowledgenewrobustcmd{\xJlength}[1]{\cmdkl{L(#1)}}

We extend the definition of "regular $\mathcal{J}$-length" to elements of the semigroup.

\begin{definition}
	The ""regular $\mathcal{J}$-length@@element"" of an element $x \in \monoid$ is the maximal number $h$ such that there exist idempotents $e_1,\ldots, e_h \in \idempotents{\monoid}$ such that $e_1 \Jlt \dots \Jlt e_h \Jleq x$.
	
	Equivalently, it is the "regular $\mathcal{J}$-length" of the semigroup obtained by restricting $\monoid$ to $\set{y \in \monoid \mid y \Jleq x}$.
	
	We write $\intro*\xJlength{x}$ for the "regular $\mathcal{J}$-length@@element" of $x$.
\end{definition}

Let us recall two facts on Green relations. For an in-depth introduction to Green relations, see for instance~\cite{Pin86a}.

\begin{lemma}\label{lem:Hone}
	In a finite semigroup $\monoid$, every equivalence class for $\Hgreen$ contains at most one idempotent.	
\end{lemma}

\begin{lemma}\label{lem:JimpliesL}
	In a finite semigroup $\monoid$, for all $x,y \in \monoid$, if $x \Jgreen y$ and $x \Lleq y$ (resp. $x \Rleq y$) then $x \Lgreen y$ (resp. $x \Rgreen y$).
\end{lemma}

We start by considering sequences of elements of $\monoid$ that remain within a layer of the semigroup of equal "regular $\mathcal{J}$-length@@element". 
We show that we can cut every such sequence into blocks, where every block is of the form $u \alpha \beta \gamma$, with $u, \alpha, \gamma$ short and $\alpha, \beta, \gamma$ evaluating to the same idempotent.

\begin{lemma}\label{def:decomposition}
	Call a ""block"" a word of the form  $u \alpha \beta \gamma$, with $\evalmorph{\alpha}= \evalmorph{\beta}= \evalmorph{\gamma} \in \idempotents{\monoid}$.
	
	Let $w \in \monoid^*$.
	A ""block decomposition"" of $w$ is a sequence of tuples 
	$(u_i,\alpha_i, \beta_i, \gamma_i)_{i\le m}$, plus an extra word $u_{m+1}$, with:
	\begin{itemize}
		\item $w = u_1\alpha_1\beta_1\gamma_1  u_2 \dots \alpha_m \beta_m\gamma_m u_{m+1}$
		\item All $u_i$, $\alpha_i$ and $\gamma_i$ have length at most $\Ramsey{\monoid}{3}$
		\item For all $i$ there exists $e_i \in \idempotents{\monoid}$ such that $\evalmorph{\monoid}(\alpha_i) = \evalmorph{\monoid}(\beta_i) = \evalmorph{\monoid}(\gamma_i) = e_i$.
	\end{itemize}
\end{lemma}

We start by cutting the word $w$ in an unbounded number of blocks, and then merge some of those blocks to obtain a bounded number of them.

\begin{lemma}\label{lem:decomposition}
	Every word  $w  = x_1 \dots x_k \in \monoid^*$ has a "block decomposition".
\end{lemma}

\newcommand{\ub}{\overline{u}}
\newcommand{\eb}{\overline{e}}
\newcommand{\Bb}{\overline{B}}
\newcommand{\alb}{\overline{\alpha}}
\newcommand{\beb}{\overline{\beta}}
\newcommand{\gab}{\overline{\gamma}}

\begin{proof}	
	Consider the following algorithm:
	\begin{algorithm}[h]
		\caption{
			Algorithm to cut the word into blocks}\label{alg:cutting1}
		\begin{algorithmic}[1]
			\State{$i,j, p \gets 1$}
			\While{$k-i > \Ramsey{\monoid}{3}$}
				\While{$x_i \dots x_j$ is not a block}
					\State $j \gets j+1$
				\EndWhile
				\State $B_p \gets x_i \dots x_j$
				\State $i, j \gets j+1$
				\State $p \gets p+1$
			\EndWhile
			\State $\ub_{p} \gets x_i \dots x_k$
		\end{algorithmic}
	\end{algorithm}

	Observe that the internal while loop must always terminate with $j\leq i+\Ramsey{\monoid}{3}$: by definition of $\Ramsey{\monoid}{3}$, if $j-i > \Ramsey{\monoid}{3}$, there must be three consecutive infixes in $x_i \dots x_j$ evaluating to the same idempotent. 
	As a consequence, $x_i \dots x_j$ must have a block as a prefix.
	
	In particular, we cannot reach the end of the word while in the internal loop: before entering it, the external loop condition guarantees that $k-i > \Ramsey{\monoid}{3}$, hence we will find a block before reaching the end of the word.
	
	Let $r$ be the value of $p-1$ at the end of the execution of Algorithm~\ref{alg:cutting1}. 
	We obtain a sequence of blocks $B_1, \dots, B_{r}$ and a suffix $B_{r+1}$ such that $x_1 \dots x_k = B_1 \dots B_{r} \ub_{r+1}$.

	Since each $B_i$ with $i\leq r$ is a block, we can define $\ub_i, \alb_i, \beb_i, \gab_i \in \monoid^*$ and $\eb_i \in \idempotents{\monoid}$ such that 
	$B_i = \ub_i \alb_i \beb_i \gab_i$ and $\evalmorph{\monoid}(\alb_i) = \evalmorph{\monoid}(\beb_i) = \evalmorph{\monoid}(\gab_i) = \eb_i$.
	Observe that all those words have length at most $\Ramsey{\monoid}{3}$.
\end{proof}	
	
\begin{lemma}
	Let $w = x_1 \dots x_k$ be such that $\xJlength{x_i} = \xJlength{\evalmorph{\monoid}(w)}$.
	Then $w$ has a "block decomposition" with at most $|\monoid| \Ramsey{\monoid}{3}$ elements.
\end{lemma}

\begin{proof}
	In all that follows, given a "block" $B = (u,\alpha, \beta, \gamma)$, we write $\Bb$ for the word $u\alpha \beta \gamma$.
	
	Let $B_1, \dots, B_m$  be a "block decomposition", with $B_i = (u_i, \alpha_i, \beta_i, \gamma_i)$. 
	We say that we can reduce it if there exist $i<j$ such that \[\evalmorph{\monoid}(\alpha_i) = \evalmorph{\monoid}(\alpha_j)= \evalmorph{\monoid}(\beta_i \gamma_i \Bb_{i+1} \dots \Bb_{j-1} u_j \alpha_j \beta_j).\]
	In that case it reduces to the "block decomposition" \[B_1, \dots, B_{i-1} , (u_i, \alpha_i, \beta_i \gamma_i  \Bb_{i+1} \dots \Bb_{j-1} u_j \alpha_j \beta_j, \gamma_j), B_{j+1}, \dots, B_k.\]
	Note that this sequence is also a "block decomposition" of $w$, with less elements than the initial one.
	
	By Lemma~\ref{lem:decomposition}, $w$ has a "block decomposition".
	We can reduce it until we cannot anymore. 
	Let $B_1, \dots, B_m$ be a "block decomposition" of $w$ that cannot be reduced, with $B_i = (u_i, \alpha_i, \beta_i, \gamma_i)$.
	
	For all $j$ let $e_j = \evalmorph{\beta_j} \in \idempotents{\monoid}$. 
	We show that $m \leq |\monoid|\Ramsey{\monoid}{3}$, by contradiction.
	Suppose $m>|\monoid|\Ramsey{\monoid}{3}$. 
	Then by the pigeonhole principle there exists $e$ an idempotent which is equal to $e_j$ for at least $\Ramsey{\monoid}{3}+1$ distinct values of $j$.
	Let $j(0) < \ldots < j(\ell)$ be such that $e_{j(i)} = e$ for all $i$ and $\ell \geq \Ramsey{\monoid}{3}$.
	
	For all $i \in [1, \ell]$ let $y_i = \gamma_{j(i-1)} \Bb_{j(i-1)+1} \dots \Bb_{j(i)-1} \alpha_{j(i)} \beta_{j(i)}$. 
	Since $\ell \geq  \Ramsey{\monoid}{3} \geq \Ramsey{\monoid}{1}$, by definition of $\Ramsey{\monoid}{1}$ there exist $i< i'$ such that $\evalmorph{\monoid}(y_i \dots y_{i'})$ is an idempotent $f \in \idempotents{\monoid}$.   
	
	We have $f =  \evalmorph{\monoid}(y_i\dots y_{i'})$ and thus 
	\begin{align*}
		f &= \evalmorph{\monoid}(\gamma_{j(i-1)}) \evalmorph{\monoid}(\Bb_{j(i-1)+1} \dots \Bb_{j(i')-1}) \evalmorph{\monoid}(\alpha_{j(i')} \beta_{j(i')})\\ 
		&= e \evalmorph{\monoid}(\Bb_{j(i-1)+1} \dots \Bb_{j(i')-1})  e.
	\end{align*}
	Hence $f <_{\Hgreen} e$. 
	
	Let $\ell$ be the "regular $\mathcal{J}$-length@@element" of $x_1 \dots x_k$ (and of all $x_i$).
	Since $e$ and $f$ are the values of infixes of $x_1 \dots x_k$, they must both have "regular $\mathcal{J}$-length@@element" $\ell$, and thus $f \Jgreen e$. 
	By Lemma~\ref{lem:JimpliesL}, we have $f \Hgreen e$. 
	Since an $\Hgreen$-class contains at most one idempotent by Lemma~\ref{lem:Hone}, we have $f=e$.
	
	This means that $e = \evalmorph{\monoid}(\gamma_{j(i-1)}\Bb_{j(i-1)+1} \dots \Bb_{j(i')-1}\alpha_{j(i')} \beta_{j(i')})$. Since $e(j(i)) = ej(i')$. 
	As a consequence, the "block decomposition" is reducible, a contradiction.
	Hence $m \leq |\monoid|\Ramsey{\monoid}{3}$.
\end{proof}

\begin{corollary}\label{cor:Jheightone}
	Let $w=x_1 \dots x_k \in \monoid^*$ be a word so that $\xJlength{\evalmorph{\monoid}(x_0 \cdots x_k)} = \xJlength{x_i}$ for all $i$.
	There exists a "summary" of $x_1 \dots x_k$ of height at most $(\log_2(|\monoid|) + 2\log_2(\Ramsey{\monoid}{3})) +3$.
\end{corollary}

\begin{proof}
	We use the decomposition from Lemma~\ref{lem:decomposition}, and the same notations.
	We use a tree of height at most $ \log(2m+1) \leq \log(2|\monoid|\Ramsey{\monoid}{3}+1)$ of product nodes to obtain one leaf for each $u_i$  and one for each $\alpha_i \beta_i \gamma_i$. Then we apply idempotent nodes for each $\alpha_i \beta_i \gamma_i$ so that all leaves are labeled by some $u_i$, $\alpha_i$ or $\gamma_i$. 
	We can then apply $\log_2(2\Ramsey{\monoid}{3})$ product nodes on each to obtain a tree where all  leaves are labeled by some $x_i$.

	In total we have built a tree of height at most $\log(2|\monoid|\Ramsey{\monoid}{3}+1) + 1+ \log_2(2\Ramsey{\monoid}{3}) \leq \log_2(|\monoid|) + 2\log_2(|\Ramsey{\monoid}{3}|) +3$.
\end{proof}

We now show how to use this lemma to prove Theorem~\ref{thm:summary}.

\begin{proof}[Proof of Theorem~\ref{thm:summary}]
	We prove by strong induction on $\ell$, the "regular $\mathcal{J}$-length@@element" of $\evalmorph{\monoid}(w)$, that $w$ has a "summary" of height at most $\ell(\log_2(|\monoid|) + 2\log_2(\Ramsey{\monoid}{3}) +4)$.

	Let $w \in \monoid^*$ such that $\evalmorph{\monoid}(w)$ has "regular $\mathcal{J}$-length@@element" $\ell$, suppose we have the lemma for all words of "regular $\mathcal{J}$-length@@element" $<\ell$.
	
	We cut $w$ into $w_1 a_1 \dots w_k a_k w_{k+1}$ so that $w_i$ is the largest prefix of $w_i a_{i} \dots w_{k+1}$ of "regular $\mathcal{J}$-length@@element" $<\ell$.
	Note that since  $w$ has  "regular $\mathcal{J}$-length@@element" $\ell$, $w_i a_i$	must have  "regular $\mathcal{J}$-length@@element" $\ell$ as well.
	
	For each $i$ let $x_i = \evalmorph{\monoid}(w_i a_i)$. All $x_i$ have the same "regular $\mathcal{J}$-length" $\ell$ as $\evalmorph{\monoid}(x_1 \dots x_k)$.
	By Corollary~\ref{cor:Jheightone} there exists a "summary" $\tau$ of $x_1 \dots x_k$ of height bounded by $ \log_2(|\monoid|) + 2\log_2(\Ramsey{\monoid}{3}) +3$.
	
	Furthermore, by induction hypothesis we have a "summary" for each $w_i$, of height at most $(\ell-1)(\log_2(|\monoid|) + 2\log_2(\Ramsey{\monoid}{3}) +4)$.
	An additional product node gives us a "summary" $\tau_i$ for $w_ia_i$, of height at most $(\ell-1)(\log_2(|\monoid|) + 2\log_2(\Ramsey{\monoid}{3}) +4) +1$.
	
	We append $\tau_i$ at the leaf corresponding to $x_i$ in $\tau$, for all $i$.
	We obtain a "summary" of $w$, of height at most  $\ell(\log_2(|\monoid|) + 2\log_2(\Ramsey{\monoid}{3}) +4)$.
	
	The induction is proven. To obtain the theorem, we observe that $\ell$ is at most $\Jlength{\monoid}$. 
	By Theorem~\ref{thm:Ramseybounds}, we have $\Jlength{\monoid} \leq \log_2(\Ramsey{\monoid}{2}) \leq \log_2(\Ramsey{\monoid}{3})$. 
	In conclusion, for every word $w \in \monoid^*$ there exists a "summary" of height at most $\log_2(\Ramsey{\monoid}{3})(\log_2(|\monoid|) + 2\log_2(\Ramsey{\monoid}{3}) +4)$.
\end{proof}

\subsection{Proof of Lemma~\ref{lem:bound-chains-sharp}}\label{app:sharp-height}

The following proof is inspired by~\cite{DBLP:journals/ita/Kirsten05}.
For this section it is convenient to visualise elements of $\F$ as graphs over $V$. Recall that $n$ is the number of vertices $|V|$.
\AP Given an idempotent element $e \in \F$, its ""$\omega$-graph"" is the directed graph whose set of vertices is $V$, with an edge from $v$ to $v'$ if and only if $e(v,v') = \omega$. 
Strongly connected components of this graph are called ""$\omega$-SCCs"" of $e$. 
An "$\omega$-SCC" $C$ is ""non-trivial"" if there is at least one edge $e(v,v') = \omega$ with $v,v' \in C$. The set of "non-trivial" "$\omega$-SCCs" is denoted $\intro*\omegaSCC{e}$.
We also define the relation $\intro*\omegaEdges{e}$ over $\omegaSCC{e}$ such that $(C_1, C_2) \in R_e$ if there exist $v_1 \in C_1$ and $v_2 \in C_2$ with $e(v_1, v_2) = \omega$.

\begin{remark}
	\label{rmk-transitive}
	The "$\omega$-graph" of an idempotent $e$ is necessarily transitive: for all $v_1, v_2, v_3 \in S$, if $e(v_1, v_2) = e(v_2, v_3) = \omega$ then, since $e \cdot e = e$, $e(v_1, v_3) = \omega$.
	As a consequence, given a "non-trivial" "$\omega$-SCC" $C$, we necessarily have $e(v,v') = \omega$ for all $v,v' \in C$ (in particular, every state in $C$ has a self-loop).
	
	Another consequence is that if $(C_1, C_2) \in \omegaEdges{e}$ then $e(v_1, v_2) = \omega$ for all $v_1 \in C_1$ and $v_2 \in C_2$.
\end{remark}

In order to bound the number of "unstable idempotent nodes" along a branch of a factorisation, we show that a certain quantity can only decrease or stay the same among idempotents along a branch, and has to decrease at every "unstable idempotent node".
The first component of that quantity is the number of "$\omega$-SCCs".

\begin{lemma}\label{lem:decrease-omega-SCC}
	For all idempotents $e,f$, if $e \Jleq f$ then $|\omegaSCC{e}| \leq |\omegaSCC{f}|$.
\end{lemma}

\newcommand{\yields}{\mathtt{yields}}

\begin{proof}
	This proof is inspired by the one of~\cite[Proposition 5.5]{DBLP:journals/ita/Kirsten05}.
	Since  $e \Jleq f$, there exist $x, y \in M$ such that $e = xfy$. We can assume without loss of generality that $x = xf$ and $y = fy$; otherwise we replace $x$ by $x' =xf$ and $y$ by $y' = fy$: we  then have $x' = x'f$, $y' = fy'$ and $e = x' f y'$.
	
	We define an injective function $\yields : \omegaSCC{e} \to \omegaSCC{f}$ as follows.
	For each "non-trivial" "$\omega$-SCC" $C$ of $e$, we select any vertex $v_C \in C$. 
	Since $C$ is "non-trivial", $e(v_C, v_C) = \omega$ by Remark~\ref{rmk-transitive}.
	We select a vertex $v$ such that $x(v_C,v) = f(v,v) = y(v,v_C) = \omega$ and define $\yields(C)$ as the "$\omega$-SCC" of $v$, which is "non-trivial" since $f(v,v) = \omega$.
	
	Let us first argue that such a $v$ always exists.
	Since $e = xfy$, there exist $v', v''$ such that $x(v_C,v') = f(v',v'') = y(v'', v_C) =\omega$.
	Hence by Lemma~\ref{lem:idempotent} there exists $\vb \in S$ such that $f(v',\vb) = f(\vb,\vb) = f(\vb,v'')$. Since $x = xf$ and $y = fy$, we obtain $x(v_C,\vb) = f(\vb,\vb) = y(\vb,v_C) = \omega$.
	
	It remains to show that this function is injective. Let $C, C'$ be two "$\omega$-SCCs" of $e$ such that $\yields(C) = \yields(C')$.
	By definition of $\yields$, we have two vertices $v \in C, v'\in C'$ and two others $\vb, \vb' \in \yields(C) = \yields(C´)$ such that $x(v,\vb) = x(v', \vb') = y(\vb,v) =y(\vb',v') = \omega$.
	Since $\vb,\vb'$ are in the same "$\omega$-SCC" of $f$, we have $f(\vb,\vb') = f(\vb', \vb) = \omega$ by Remark~\ref{rmk-transitive}.
	As a consequence, since $e = xfy$, we obtain $e(v,v') = e(v',v) = \omega$, hence $v$ and $v'$ are in the same "$\omega$-SCC" of $e$.
\end{proof}

The second component of that quantity is the size of the relation $\omegaEdges{e}$ over "$\omega$-SCCs".

\begin{lemma}\label{lem:decrease-omega-reach}
	For all idempotents $e$ and $f$ such that $e \Jleq f$ and $|\omegaSCC{e}|  = |\omegaSCC{f}|$, we have $|\omegaEdges{e}| \geq |\omegaEdges{f}|$.
\end{lemma}

\begin{proof}
	If $|\omegaSCC{e}|  = |\omegaSCC{f}|$ then the function $\yields$ defined in the proof of Lemma~\ref{lem:decrease-omega-SCC} is actually a bijection. 
	
	We show that for all $(C,C')$ in $\omegaEdges{f}$, $(\yields^{-1}(C),\yields^{-1}(C')) \in \omegaEdges{e}$, which implies the lemma. Let $(C,C')$ in $\omegaEdges{f}$.	
	Let $v \in \yields^{-1}(C)$ and $v' \in \yields^{-1}(C')$. 
	By definition of $\yields$, there exist  $\vb \in C$ and $\vb' \in C'$ such that $x(v,\vb) = f(\vb,\vb) = y(\vb,v)$ and  $x(v',\vb') = f(\vb',\vb') = y(\vb',v')$. 
	
	Furthermore, since $(C,C') \in \omegaEdges{f}$, there exist $v'' \in C, v''' \in C'$ such that $f(v'', v''') = \omega$.
	In light of Remark~\ref{rmk-transitive}, and since both $\vb, v''$ and $\vb', v'$ share the same SCC, we obtain $f(\vb,\vb') = \omega$, and thus $e(v,v') = \omega$.
	As a result,  we have $(\yields^{-1}(C),\yields^{-1}(C')) \in \omegaEdges{e}$.
\end{proof}

\begin{lemma}\label{lem:sharp-decrease}
	For all idempotent $e \in M$, if $e \neq e^\sharp$ then either:
	\begin{itemize}
		\item $|\omegaSCC{e^\sharp}|  < |\omegaSCC{e}|$, or
		
		\item $|\omegaSCC{e^\sharp}|  = |\omegaSCC{e}|$ and $|\omegaEdges{e^\sharp}| > |\omegaEdges{e}|$
	\end{itemize} 
\end{lemma}

\begin{proof}
	For all $v,v' \in S$, if $e(v,v') = \omega$ then $e^\sharp(v,v') = \omega$. In other words, the "$\omega$-graph" of $e$ is a subgraph of the one of $e^\sharp$, and since both $e$ and $e^\sharp$ are idempotent, both graphs are transitive.
	As a consequence, $|\omegaSCC{e}| \geq |\omegaSCC{e^\sharp}|$ and if $|\omegaSCC{e}| = |\omegaSCC{e^\sharp}|$, then  actually the two graphs have exactly the same SCCs. 
	In that case, since $e \neq e^\sharp$, there exist $v, v_0, v'_0, v'$ such that $e(v,v_0) = e(v'_0, v') = \omega$ and $e(v_0,v_0') =1$.
	By Lemma~\ref{lem:idempotent}, we have $v_1, v_1'$ such that $e(v,v_1) = e(v_1, v_1) = e(v_1, v_0) = \omega$ and $e(v'_0,v'_1) = e(v'_1, v'_1) = e(v'_1, v') = \omega$.
	Since $e$ is idempotent and $e(v_0, v'_0) = 1$, we have $e(v_1, v'_1) \leq 1$ and similarly since $e^\sharp(v_0, v'_0) = \omega$ we have $e^\sharp(v_1, v'_1) = \omega$.
	Hence the "$\omega$-graph" of $e^\sharp$ must have an extra edge between "non-trivial" "$\omega$-SCCs".
	Since the graphs are transitive, this extra edge has to be from an "$\omega$-SCC" $C$ to another $C'$ such that there are no edges from $C$ to $C'$ in $e$. 
	As a result, $|\omegaEdges{e^\sharp}| > |\omegaEdges{e}|$. 
\end{proof}

\begin{proof}[Proof of Lemma~\ref{lem:bound-chains-sharp}]
	Let $e_1, \ldots, e_m$ be idempotents in $\F$ such that $e_i^\sharp \Jleq e_{i+1}$ for all $i$.
	For each $i$ we define \[(k_i, p_i) = (n - |\omegaSCC{e_i}|, |\omegaEdges{e_i}|).\] 
	
	Observe that $k_i$ ranges from $0$ to $n$, since the number of "$\omega$-SCCs" cannot exceed the number of nodes. 
	As for $p_i$, it ranges from $0$ to $n-1$: if we had $p_i > n-1$, then also $p_i > |\omegaSCC{e_i}|-1$ and the relation $\omegaEdges{e_i}$ would have a cycle. This is impossible as we would then have two distinct "$\omega$-SCCs" connected in both directions.

	By Lemmas~\ref{lem:decrease-omega-SCC} and \ref{lem:decrease-omega-reach}, $(k_i, p_i)_{1 \leq i \leq k}$ is non-increasing for the lexicographic ordering.
	As a result, it can strictly decrease only $n(n-1) \leq n^2 - 1$ times. 
	As a consequence, there exists $i$ such that $(k_i, p_i) = (k_{i+1}, p_{i+1})$, 
	implying that $e_i = e_i^\sharp$ by Lemma~\ref{lem:sharp-decrease}. 
\end{proof}

\subsection{Proof of Theorem~\ref{cor:sharpexpression}}
\label{app:bound-sharp-exp-height}

	Clearly by definition of $\F$ every element is generated by a "$\sharp$-expression".
	Define the \emph{$\sharp$-height} of a "$\sharp$-expression" as the maximal number of idempotent nodes along a branch.
	To begin with, observe that every element of $\F$ is generated by a "$\sharp$-expression" of $\sharp$-height at most $n^2-1$: 
	Take a "$\sharp$-expression" $E$ of $\sharp$-height $n^2$ or more.
	There are idempotent nodes $\nu_1, \dots, \nu_{n^2}$ along a branch of $E$ (from leaf to root), and idempotents $e_1, \ldots, e_{n^2}$ such that $\nu_i$ is labeled $e_i^\sharp$ and has a single child labeled $e_i$.
	
	Notice that in $E$, if a node labeled $y$ is the child of a node labeled $z$, then $z \Jleq y$. It is trivial for a product node. For an idempotent node, it follows from the fact that
	$e^\sharp \cdot e= e^\sharp$, which can be checked by a simple computation.
	It follows by transitivity of $\Jleq$ that this still holds when $y$ labels a descendant of $z$, and thus that $ e_{i+1}  \Jleq e_i^\sharp $ for all $i$. 
	
	Hence there is a $\nu_i$ such that $e_i = e_i^\sharp$ by Lemma~\ref{lem:bound-chains-sharp}.
	As a result, we can reduce $E$ by replacing the subtree rooted in $\nu_i$ by the one rooted in its child. 
	
	We have shown that all elements of $\F$ are generated by a "$\sharp$-expression" of $\sharp$-height at most $n^2-1$.
	We now prove that for all $h>0$, if an element has a "$\sharp$-expression" of $\sharp$-height at most $h$ then it has one of height at most $(2n^2+1)h$.
	Take $x$ generated by $E$ of $\sharp$-height $h$, assume we have shown the property for all lower heights.
	Then $x = x_1 \dots x_m$ with for all $i$ either $x_i = x_a$ for some $a \in A$ or $x_i = y_i^\sharp$ with $y_i$ generated by a "$\sharp$-expression" of $\sharp$-height $<h$. 
	By induction hypothesis, every $x_i$ in the second case can be generated by a "$\sharp$-expression" $E_i$ of height $\leq (2n^2+1)(h-1)$.
	
	We can assume that $m \leq |\F| \leq 3^{n^2}$. Otherwise, there must exist $i<j$ with $x_1 \cdots x_i = x_1 \cdots x_j$ and we can shorten $x$ as $x = x_1 \cdots x_{i} x_{j+1} \cdots x_m$.
	Thus, we can obtain $x$ from $x_1, \ldots, x_m$ with a tree of product nodes of height at most $\log_2(|\F|) \leq 2 n^2$.
	The leaves of this tree are labeled by $x_1, \ldots, x_m$.
	For all $x_i$ that is not in $\set{x_a \mid a \in A}$, there exist $y_i$ generated by $E_i$ such that $x_i = y_i^\sharp$. 
	We append $E_i$ below the corresponding leaf labeled $x_i$.
	We obtain a "$\sharp$-expression" of height at most $(2n^2+1)(h-1) + 2n^2 +1 = (2n^2+1)h$.
	
	This concludes our induction. 
	To obtain the corollary, it suffices to apply it with $h = n^2-1$: every element has a "$\sharp$-expression" of height at most $(2n^2+1)(n^2-1) \leq 2n^4$. 

\subsection{Proof of Theorem~\ref{thm:tsef}}
\label{app:tsef}

	This proof is inspired by~\cite[Lemma 10]{simontropical}.
	We start by observing that a branch of a "$\sharp$-summary" $\tau$ can have at most $n^2-1$ "unstable idempotent nodes". 
	Suppose the contrary, let $\nu_1, \ldots, \nu_{n^2}$ be "unstable idempotent nodes" along a branch, from leaf to root, with each $\nu_i$ labeled by $e_i^\sharp$ and with a single child labeled by $e_i$.
	
	Analogously to the proof of Theorem~\ref{cor:sharpexpression}, if a node labeled $y$ is the descendant of a node labeled $z$, then $z \Jleq y$. Hence $ e_{i+1}  \Jleq e_i^\sharp $ for all $i$. 
	By Lemma~\ref{lem:bound-chains-sharp}, one of those nodes must be "stable@@node", a contradiction.

	We now prove that
	every word $w$ admits a "$\sharp$-summary" where along every branch, sequences of consecutive nodes without any leaf or "unstable idempotent node" have length at most $\shortendfactorizationupperbound -1$.
	
	By induction on $|w|$.
	Let $w$ be a word and $\tau$ a "summary" of $w$ of height at most $\shortendfactorizationupperbound$ (which exists by Theorem~\ref{thm:summary}).
	We distinguish two cases:\\
	\emph{Case 1:}
	all idempotent nodes in $\tau$ are "stable@@node". Then $\tau$ is a "$\sharp$-summary" and the property holds trivially since $\tau$ has height at most $\shortendfactorizationupperbound$.
	\\
	\emph{Case 2:}
	$\tau$ has an idempotent node $\nu$ labeled $(v,e)$ with $e^\sharp \neq e$. If $\tau$ has several such nodes, we pick one of maximal depth.
	Consider the subtree $\tau_\nu$ rooted at $\nu$. 
	Let $\tau'_v$ be the tree obtained from $\tau_\nu$ by replacing the label of $\nu$ with $(v, e^\sharp)$. 
	It is a "$\sharp$-summary" as we took $\nu$ of maximal depth and thus every idempotent node below it is "stable@@node".
	Further, it has height at most $\shortendfactorizationupperbound$, since $\tau_\nu$ is a subtree of $\tau$.
	Now consider the word $w'$ where the factor $v$ has been replaced with a single letter $e^\sharp$. It is shorter than $w$, thus it admits a "$\sharp$-summary" $\tau'$ satisfying the property. 
	To obtain the desired "$\sharp$-summary" for $w'$, we start from $\tau'$ and replace the $e^\sharp$ leaf with $\tau_v$.
	The result is a "$\sharp$-summary", which satisfies the property since the tree below $\nu$ has height at most $\shortendfactorizationupperbound$, $\nu$ is an "unstable idempotent node" and the rest of the tree satisfies that property.

	We can now bring in the bound on the number of "unstable idempotent nodes" along a branch:
	since every branch has at most $n^2 -1$ "unstable idempotent nodes", 
	and the length of sequences of nodes without those is bounded by $\shortendfactorizationupperbound -1$, the height is at most $\tropicalshortendfactorizationupperbound$.

\subsection{Proof of Theorem~\ref{thm:boundonflow}}
\label{app:bound-on-flow}

	The proof makes use of the usual notion of \emph{cuts}.

	\begin{definition}
		Given two sets of vertices $V_s$ and $V_t$, a ""cut"" between $V_s$ and $V_t$ in a "capacity word" $\capa = a_1 \dots a_k$ of length $k$ is a sequence $\cut = c_1 \dots c_k$ of sets of edges such that for every "path@@pipeline" $v_0 \dots v_k$ in $\capa$ with $v_0 \in V_s$ and $v_k \in V_t$ there exists $i$ such that $(v_{i-1},v_i) \in c_i$.
		
		The ""cost"" of this "cut" is the sum of the capacities of its edges: \[\intro*\cost(\cut) = \sum_{i=1}^k \sum_{(v,v' \in c_i)} a_i(v,v').\] 
	\end{definition}
	
	By the max flow-min cut theorem~\cite{ford1956maximal}, we know that in a "capacity word" $\capa$ the minimal "cost" of a  "cut" between two vertices  is exactly the maximal "value" of a flow between those vertices.

	Set $n=|V|$.
	We prove the following statement by induction on $h$ the height of a "$\sharp$-summary" $\tau$ of $x_{a_1}\ldots x_{a_m}$:
	there is a "cut" between $s$  and $t$ in $w=a_1\ldots a_m$ of "cost" at most $Kn^h$.
	\begin{enumerate}
		\item if $x(s,t) = 0$ then there is no "path@@pipeline" in $w$ from $s$ to $t$, and
		\item if $x(s,t) \leq 1$ then there is a "cut" between $s$ and $t$ in $w$ of "cost" at most $Kn^h$.
	\end{enumerate}

	If $h=0$ then $\tau$ is just a leaf and therefore $w$ is a single letter $a$. Since $x(s,t) \leq 1$, we have $a(s,t) < \omega$ and thus we have a "cut" of value $\leq  K$ between $s$ and $t$.
	
	If the root of $\tau$ is a product node, let $(y_1,u_1),(y_2,u_2) \in \F\times \F^*$ be the labels of its two children.
	Then $x=y_1\cdot y_2$ and $u_1 = x_{a_1}\ldots x_{a_\ell}$ and $u_2 = x_{a_{\ell+1}}\ldots x_{a_k}$.
	Since  $x(s,t) \leq 1$, for all $v \in V$, we have $y_1(s, v) \leq 1$ or $y_2(v, t) \leq 1$.
	By induction hypothesis, this means that for all $v \in V$, we have a "cut" of value at most $n^{h-1}$ between $s$ and $v$ in $a_1 \ldots a_\ell$ or between $v$ and $t$ in $ a_{\ell+1} \ldots a_k$. 
	We take the union of all these cuts. This forms a "cut" between $s$ and $t$ in $w$, of value at most $n K n^{h-1} \leq K n^h$.
	
	If the root of $\tau$ is an idempotent node, it has two children labeled $(e,u_1),(e,u_k)$
	with  $u_1, u_k \in \F^*$ and $e \in \idempotents{\F}$. 
	Then $x=e^\sharp$ and $u_1 = x_{a_1}\ldots x_{a_\ell}$ and $u_2 = x_{a_{\ell'}}\ldots x_{a_k}$ for some $1\leq \ell < \ell' \leq k$.
	By induction hypothesis, for all $v \in V$ such that $e(s,v) \leq 1$ we have a "cut" in $a_1 \ldots a_\ell$ between $s$ and $v$ of value $\leq K n^{h-1}$. 
	Similarly, for all $v' \in V$ such that $e(v',t) \leq 1$ then we have a "cut" in $a_{\ell'} \ldots a_k$ between $v'$ and $t$ of value $\leq K n^{h-1}$.
	
	Consider the union of those "cuts": its value is at most $2 Kn^{h-1} \leq Kn^h$ (we set aside the very specific and trivial case $n=1$). We only have to prove that it is a "cut" between $s$ and $t$ in $w$. 
	Since $x$ is the root of $\tau$ then $x=e^\sharp$ for some idempotent, hence $x^\sharp = x$ (cf. Lemma~\ref{lem:sharpsharp}). In particular, $(s,t)$ is a stable node in $x$.
	According to Lemma~\ref{lem:idempotent},
	every "path@@pipeline" between $s$ and $t$ must either be in a vertex $v$ with $e(s,v) \leq 1$ after going through $u_1$, or be in a vertex $v'$ with $e(v',t) \leq 1$ before going through $u_k$.
	As a consequence, it must go through the constructed "cut".
	
	According to Theorem~\ref{thm:tsef}, the height can be bounded by $\tropicalshortendfactorizationupperbound$, hence the result.
	This ends our proof.

\section{Regular constraints}\label{app:reg}

The aim of this section is to prove Theorem~\ref{theo:regularSFP}. 

\AP We extend the notion of "fair unboundedness witness": 
A ""fair unboundedness witness for $\mathcal{A}$"" is an element $(q,x,q') \in \FA$ such that $q \in I$, $q' \in F$ and $x(v_s, v_t) = \omega$ for all $(v_s, v_t) \in E$,

\begin{lemma}[Adapted from Lemma~\ref{lem:sufficientwitness}]
	\label{lem:sufficientwitnessreg}
	If the "labeled flow semigroup" contains an element $(q_i,x,q_f)$ with $q_i \in I$, $q_f \in F$ and $x(v_s,v_t) = \omega$ for all $(v_s, v_t) \in E$ then the answer to the \rmsfp\ is $\om$. 
\end{lemma}
\begin{proof}
	We show that all elements of $\FA$ besides $\bot$ satisfy the following property:
	For all $(q,x,q') \in \FA$, for all $N \in \NN$, there exist a "capacity word" $\capa$ and a "token flow" $\dsf$ over $\capa$ such that $\capa$ labels a run of $\mathcal{A}$ from $q$ to $q'$ and for all $v, v' \in V$, the following two conditions are satisfied:
	\begin{enumerate}
		\item $x(v,v') = \omega$ $\Rightarrow$ $\gsf(\dsf)(v,v') \geq N$
		
		\item $x(v,v') \geq 1$ $\Rightarrow$   there is a "path@@pipeline" in $\capa$ from $v$ to $v'$ 
	\end{enumerate}
The proof is then a straightforward adaptation of Appendix~\ref{app:sufficient}.
\end{proof}

For the other direction, we need to bound the "regular $\mathcal{J}$-length" of $\FA$. Although bounds on the "Ramsey function" can be inferred through Theorem~\ref{thm:Ramseybounds}, we obtain better bounds by a simple pigeonhole argument.

\begin{lemma}\label{lem:boundsFA}
	The "labeled flow semigroup" satisfies the following inequalities: 
	\begin{enumerate}
		\item 
		$|\FA| \leq m^2|\F| +1 \leq m^2 3^{n^2} +1$,
		
		\item 
		$\Jlength{\FA} \leq \Jlength{\F} +1 \leq (n^2+n+2)^2/4 +1$, and
		
		\item for all $k \in \NN$ ,
		$\Ramsey{\FA}{k} \leq \Ramsey{\F}{k(m+1)}$. In particular, $\Ramsey{\FA}{3} \leq (3(m+1)|\F|^4)^{\Jlength{\F}} \leq (m+1)^{16n^4} 3^{32n^6}$ 
	\end{enumerate}
\end{lemma}

\begin{proof}
	\begin{enumerate}
		\item The bound on the size follows from the simple observation  that for every element of $\FA$ of the form $(q,x,q')$, by definition of $\FA$, $x$ can be obtained from the $(x_a)_{a\in A}$ using product and $\sharp$. Hence $x \in \F$, and thus $|\FA| \leq |Q|^2 |\F| +1$. 
		
		\item It was shown by Jecker that the "regular $\mathcal{J}$-length" $\Jlength{\monoid}$ of a finite semigroup $\monoid$ is equal to the largest $m$ such that the max-monoid $H_m$ can be embedded in $\monoid$~\cite[Appendix B]{JeckerArxiv}. We use this to show that the "regular $\mathcal{J}$-length" of $\FA$ is at most the one of $\F$ plus one.
		
		Let $m \in \NN$, suppose we have an injective morphism $\psi: H_{m+1} \to \FA$.
		At most one element can be mapped to $\bot$, and if there is one it must be $m+1$.
		We exhibit an injective morphism $H_m \to \F$.
		
		For all $i \in [1,m]$, we must have $\psi(i) \star \psi(i) = \psi(\max(i,i)) = \psi(i)$, hence all $\psi(i)$ are idempotents, of the form $(q_i, e_i, q_i)$.
		Further, for all $i<j$ we have $\psi(i) \star \psi(j) = \psi(\max(i,j)) = \psi(j)$, hence $q_i = q_j$.
		As a result, there is a state $q$ such that $\psi(i) = (q,e_i,q)$ for all $i$.
		Since $\psi$ is injective, all $e_i$ must be distinct.
		
		It suffices to observe that the function $\psi': H_m \to \F$ with $\psi'(i) = e_i$ is an injective morphism.
		
		As a consequence, $\Jlength{\FA} \leq \Jlength{\F}+1$.
		
		\item For the "Ramsey function", we use a simple pigeonhole argument. 
		Let $\pi : \FA \to \F$ be the morphism projecting each triple $(q,x,q') \in \FA$ to its middle component $x \in \F$.
		Let $k \in \NN$, let $w \in \FA^*$ be a word of length $k \Ramsey{\F}{k|Q|}$.
		We show that $w$ contains $k$ consecutive infixes evaluating to the same idempotent.
		
		We first cut $w$ in $k$ parts of length $\Ramsey{\F}{k|Q|}$: $w = w_1 \dots w_k$. If $\evalmorph{\FA}(w_j) = \bot$ for all $j$, then we have the desired consecutive infixes.
		Otherwise, there is some $p$ such that $\evalmorph{\FA}(w_j) \neq \bot$.
		
		Since $w_p$ has length at least $\Ramsey{\F}{k|Q|}$, by definition of the "Ramsey function" it contains an infix $u_1\dots u_{k(|Q|)}$ such that there is an idempotent $e \in \F$ with $\evalmorph{\F}(\pi(u_i)) = e$ for all $i$. 
		Since $\evalmorph{\FA}(w_j) \neq \bot$, there are states $q_0, \ldots, q_{k|Q|}$ such that we have $\evalmorph{\F}(\pi(u_i)) = (q_i, e_i, q_{i+1})$ for all $i$.
		By the pigeonhole principle, there exist $i_0 < \dots < i_k$ such that $q_{i_0} = \dots = q_{i_k}$. Let $q$ be that state.
		As a result, the consecutive infixes $u_{i_j}\dots u_{i_{j+1} -1}$ all evaluate to the same idempotent $(q,e,q)$.
		\qedhere
	\end{enumerate}
\end{proof}

\begin{lemma}[Adapted from Theorem~\ref{thm:boundonflow}]\label{thm:boundonflowreg}
	For all "capacity words" $w$, for all $q,q' \in Q$, if there is a run reading $w$ from $q$ to $q'$ in $\mathcal{A}$ then there exists $x$ such that $(q,x,q') \in \FA$ and for all $v,v' \in V$,
	\begin{align*}
		x(v,v') = 0 &\implies \gsf(\dsf)(v,v') = 0 \qquad\text{and} \\
		x(v,v') = 1 &\implies \gsf(\dsf)(v,v') \leq K(2|V|)^{(170\log_2(m)+835)n^{12}}.
	\end{align*}
\end{lemma}

\begin{proof}
	To begin with, we can apply Theorem~\ref{thm:summary} to $\FA$: 
	
	Every word $w \in \FA^*$ has a "summary" of height at most \[\Jlength{\FA} (\log_2(|\FA|) + 2\log_2(\Ramsey{\FA}{3}) +4).\] 
	By Lemma~\ref{lem:boundsFA}, this is bounded by 
	
	\begin{align*}
		&(\frac{(n^2+n+2)^2}{4}+1) (\log_2(|Q|^2 3^{n^2}+1) + 2\log_2((|Q|+1)^{16n^4}3^{32n^6}) +4) \\
		\leq &(5n^4)(2\log_2(|Q|) + 2n^2 +1 + 32n^4 \log_2(|Q|) + 32n^4 + 128n^6 + 4 )\\
		\leq &(170\log_2(|Q|) + 835)n^{10}
	\end{align*}
	
	We extend the "$\sharp$-summaries" to the "labeled flow semigroup".
	A ""labeled $\sharp$-summary"" is defined analogously to a "$\sharp$-summary", but the labels are in $\FA \times \FA^*$ instead of $\F \times \F^*$.
	
	Lemma~\ref{lem:bound-chains-sharp} can be extended to  $\FA$: 
	Let $(q_1, e_1, q_1), \dots, (q_{n^2}, e_{n^2}, q_{n^2})$ be idempotents of $\FA$ (different from $\bot$) such that \[(q_i, e_i, q_i)^{\sharp} \Jleq (q_{i+1}, e_{i+1}, q_{i+1})\] for all $i$.
	Then, for all $i$ we have $e_{i}^\sharp \Jleq e_{i+1}$. 
	By  Lemma~\ref{lem:bound-chains-sharp}, there exists $i$ such that $e_i^\sharp = e_i$ . 
	As a consequence, $(q_i, e_i, q_i)^\sharp = (q_i, e_i^\sharp, q_i) = (q_i, e_i, q_i)$. 
	
	Thus, we can repeat the proof of Theorem~\ref{thm:tsef} for $\FA$.
	We obtain that each word has a "labeled $\sharp$-summary" of height at most $(170\log_2(|Q|) + 835)n^{12}$.
	
	This lets us in turn adapt Theorem~\ref{thm:boundonflow}: if $K$ is  the largest finite coordinate in capacities, and $w$ has a "labeled $\sharp$-summary" whose result is $x$ with $x(s,t) \leq 1$ then the maximal flow in $w$ from $s$ to $t$ is bounded by $K(2n)^{(170\log_2(|Q|) + 835)n^{12}}$.
\end{proof}

This leads to the following corollary which allows to compute the optimal sequential flows in polynomial space just like in the previous cases.

\begin{lemma}[Adapted from Lemma~\ref{lem:necessarywitness}]\label{lem:necessaryreg}
	If the answer to the \rmsfp\ is $\omega$ then there is a "fair unboundedness witness for $\mathcal{A}$" in $\FA$.
\end{lemma}

\begin{proof}
	Since we have "capacity words" in $L$ with "token flows" of unbounded values, in particular there exists a "capacity word" $\capa \in L$ with a "token flow" $d$ transferring more than $K(2|V|)^{(170\log_2(m)+835)n^{12}} $ tokens between each pair of vertices in $E$.
	Let $q \in I,q' \in F$ such that $w$ labels a run from $q$ to $q'$ in $\mathcal{A}$.
	We apply Lemma~\ref{thm:boundonflowreg} to $d$.
	By case inspection, the only possible value of $x(v,v')$
	is $\omega$, thus $(q,x,q')$ is a "fair unboundedness witness for $\mathcal{A}$".
\end{proof}

\end{document}